\documentclass{amsart}

\usepackage{graphicx}

\usepackage{amsmath,amsfonts,amssymb,float,subfig,mathrsfs,hyperref,pdflscape}

\newcommand{\paren}[1]{\ensuremath{\left( #1 \right)}}
\newcommand{\set}[1]{\ensuremath{\left\{ #1 \right\}}}
\newcommand{\cyclic}[1]{\ensuremath{\left< #1 \right>}}
\newcommand{\braces}[1]{\ensuremath{\left[ #1 \right]}}
\newcommand{\norm}[1]{\ensuremath{\left\| #1 \right\|}}
\newcommand{\abs}[1]{\ensuremath{\left| #1 \right|}}

\newcommand{\Matrix}[1]{\begin{pmatrix}#1\end{pmatrix}}
\newcommand{\SmallMatrix}[1]{\left(\begin{smallmatrix}#1\end{smallmatrix}\right)}

\DeclareMathOperator*{\res}{res}

\newcommand{\BetaFun}[5]{B\paren{\begin{matrix}\begin{matrix}#1,&#2\end{matrix}\\#3\end{matrix} \, ;\, \begin{matrix}#4\\#5\end{matrix}}}

\renewcommand{\Re}{{\mathop{\mathgroup\symoperators Re}}}
\renewcommand{\Im}{{\mathop{\mathgroup\symoperators Im}}}
\newcommand{\sgn}{{\mathop{\mathgroup\symoperators \,sign}}}

\newcommand{\Max}[1]{\ensuremath{\max \set{#1}}}
\newcommand{\Min}[1]{\ensuremath{\min \set{#1}}}
\newcommand{\BigO}[1]{O\paren{#1}}

\newcommand{\Z}{\mathbb{Z}}
\newcommand{\R}{\mathbb{R}}
\newcommand{\N}{\mathbb{N}}

\newcommand{\C}{\mathbb{C}}

\newcommand{\wbar}[1]{\overline{#1}}
\newcommand{\wtilde}[1]{\widetilde{#1}}
\newcommand{\e}[1]{e\paren{#1}}
\newcommand{\pFqName}[2]{{_{#1}F_{#2}}}
\newcommand{\pFq}[5]{\pFqName{#1}{#2}\paren{\begin{array}{c}#3;\\#4;\end{array}#5}}
\newcommand{\trans}[1]{{^t #1}}

\newcommand{\pderv}[3][]{{\frac{\partial^{#1}#2}{\partial{#3}^{#1}}}}

\theoremstyle{plain} %% This is the default
\newtheorem{thm}{Theorem}%%[section]
\newtheorem{cor}[thm]{Corollary}
\newtheorem{lem}[thm]{Lemma}
\newtheorem{prop}[thm]{Proposition}
\newtheorem{conj}[thm]{Conjecture}

\theoremstyle{definition}

\theoremstyle{remark}

\newif\ifshowTODOs
\showTODOsfalse

\newcommand{\thmref}[1]{Theorem~\ref{#1}}

\newcommand{\lemref}[1]{Lemma~\ref{#1}}
\newcommand{\propref}[1]{Proposition~\ref{#1}}

\newcommand{\tableref}[1]{Table~\ref{#1}}

\newcommand{\mathdash}{\mbox{--}}

\begin{document}

\title{On sums of $SL(3,\Z)$ Kloosterman sums.}

\author{Jack Buttcane}
\date{\today}

\thanks{Email address: buttcane@uni-math.gwdg.de}
\thanks{During the time of this research, the author was supported by National Science Foundation DMS-10-01527 and VIGRE grants.}

\begin{abstract}
We show that sums of the $SL(3,\Z)$ long element Kloosterman sum against a smooth weight function have cancellation due to the variation in argument of the Kloosterman sums, when each modulus is at least the square root of the other.
Our main tool is Li's generalization of the Kuznetsov formula on $SL(3,\R)$, which has to date been prohibitively difficult to apply.
We first obtain analytic expressions for the weight functions on the Kloosterman sum side by converting them to Mellin-Barnes integral form.
This allows us to relax the conditions on the test function and to produce a partial inversion formula suitable for studying sums of the long-element $SL(3,\Z)$ Kloosterman sums.
\end{abstract}

\maketitle

\section{Introduction}
\label{sec:Intro}
The classical Kloosterman sums originated in 1926 in the context of applying the circle method to counting representations of integers by the four-term quadratic form $ax^2+by^2+cz^2+dt^2$ \cite{Kl01}; they are defined by
\[ S(a,b,c) = \sum_{\substack{x \pmod{c} \\ (x,c) = 1 \\ x \bar{x} \equiv 1 \pmod{c}}} \e{\frac{ax+b\bar{x}}{c}}, \qquad \e{x} = e^{2\pi i x}, \]
and they enjoy a multiplicativity relation:
If $(c,c')=1$, then 
\[ S(a,b,cc') = S(\wbar{c'}a,\wbar{c'}b,c) S(\wbar{c}a,\wbar{c}b,c'), \]
where $c'\wbar{c'}\equiv 1 \pmod{c}$, $c\wbar{c}\equiv 1 \pmod{c'}$.
In 1927, Kloosterman \cite{Kl02} used these sums to estimate Fourier coefficients of modular forms, as did Rademacher in 1937 \cite{Rad01}.
Optimal estimates for individual Kloosterman sums were obtained in 1948 by Andr\'e Weil \cite{W01}: $\abs{S(a,b,c)} \le d(c) \sqrt{(a,b,c)} \sqrt{c}$, where $d(c)$ is the number of positive divisors of $c$ and $(a,b,c)$ is the greatest common divisor.
In 1963, Linnik published a paper outlining methods for problems in additive number theory \cite{L01} in which he noted the importance of sums of Kloosterman sums and made the conjecture that such sums should have good cancellation between terms:
\begin{conj}[Linnik]
	Let $N$ be large and $C > N^{\frac{1}{2}-\epsilon}$, then
	\[ \sum_{c\le C} S(1,N,c) \ll C^{1+\epsilon}. \]
\end{conj}
One should compare this to Weil's estimate which gives $C^{\frac{3}{2}+\epsilon}$.

On a parallel track, between 1932 and 1940, Petersson \cite{Pe01}, Rankin \cite{Ran01} and Selberg \cite{Se01} connected Fourier coefficients of modular forms to sums of Kloosterman sums by studying Poincar\'e series.
This led to Kuznetsov's trace formulas which relate sums of Kloosterman sums to sums of Fourier coefficients of $SL(2,\Z)$ automorphic forms, and using these formulas in 1980, Kuznetsov \cite{Kuz01} was able to make progress towards Linnik's conjecture:
\begin{thm}[Kuznetsov]
\label{thm:KuznetsovsBound}
	\[ \sum_{c \le T} \frac{1}{c} S(n,m,c) \ll_{n,m} T^{\frac{1}{6}} \paren{\ln T}^{\frac{1}{3}}. \]
\end{thm}
As Weil's estimate here gives $T^{\frac{1}{2}+\epsilon}$, we must be seeing cancellation between terms as Linnik predicted.

This second track has been quite fruitful for the followers of Iwaniec -- sums of arithmetic functions, usually related to quadratic forms in some sense, can sometimes be decomposed into sums of Kloosterman sums, e.g. \cite{DI03}, and similarly, exponential sums related to quadratic forms can often be decomposed into Poincar\'e series, e.g. \cite{DFI02}.
The Kuznetsov trace formulas then play the role of Poisson summation, allowing one to substitute a sum of Fourier coefficients of automorphic forms for a sum of Kloosterman sums and visa versa.
Iwaniec in particular has made good use of a sort of double application of Kuznetsov's formulas; using positivity to study averages of Fourier coefficients of automorphic forms via the Kuznetsov formula and then applying these estimates to sums of Kloosterman sums via the second form of the Kuznetsov formula, e.g. \cite{DI03}.

Finally, we note that the Fourier coefficients of automorphic forms which are also eigenfunctions of the Hecke operators give rise to $L$-functions.
By applying the Kuznetsov formulas in this situation we may obtain results on averages of $L$-functions and all of the problems to which such things apply, e.g. \cite{DFI03}.

Now having noted the strong connection between analysis on $SL(2,\R)$ and quadratic forms, it is hoped that analysis on $SL(3,\R)$ will play a similar role in the study of cubic forms, and the analysis of Hecke operators on $SL(3,\R)$ automorphic forms is also known to give rise to $L$-functions.
A paper of Jacquet, Piatetski-Shapiro and Shilika \cite{JPS01} and a book of Bump \cite{B01} (which is essentially his dissertation), form the foundations of the $L$-function approach; and a paper of Bump, Friedberg and Goldfeld \cite{BFG01} initiates the study of Poincar\'e series and Kloosterman sums on $SL(3,\Z)$.

The paper \cite{BFG01} notes that the Fourier coefficients of Poincar\'e series are given by sums of two new types of exponential sums in addition to the classical sums of Kloosterman himself; we will primarily be concerned with the long-element sum which we denote $S_{w_l}(\psi_m,\psi_n,c)$, for reasons which will be made clear later, and is given by the sum
\begin{align*}
	& S_{w_l}(\psi_{m_1,m_2}, \psi_{n_1,n_2}, (A_1, A_2)) =\\
	& \mathop{{\sum}^*}_{\substack{B_1, C_1 \pmod{A_1}\\B_2, C_2\pmod{A_2}}} \e{m_2 \frac{Z_2 B_1-Y_2 A_1}{A_2}+m_1 \frac{Y_1 A_2-Z_1 B_2}{A_1}+n_2\frac{B_1}{A_1}+n_1\frac{-B_2}{A_2}},
\end{align*}
here the sum $\sum^*$ is restricted to those quadruples of $B_1,C_1,B_2,C_2$ satisfying
\[ (A_1, B_1, C_1) = (A_2, B_2, C_2) = 1, \quad A_1 C_2 + B_1 B_2 + C_1 A_2 \equiv 0 \pmod{A_1 A_2}, \]
and the numbers $Y_1, Z_1, Y_2, Z_2$ are defined by
\[ Y_1 B_1 + Z_1 C_1 \equiv 1 \pmod{A_1}, \qquad Y_2 B_2+Z_2 C_2 \equiv 1 \pmod{A_2}. \]

In \cite{BFG01}, the authors list a number of basic properties of this new Kloosterman sum, which generally relate to its well-definedness and interchanging indices of characters or moduli, but the most important is a type of multiplicativity:
\begin{lem}[BFG]
	If $(c_1c_2,c_1'c_2')=1$ and
	\[ \wbar{c_1}c_1 \equiv \wbar{c_2}c_2 \equiv 1 \pmod{c_1' c_2'}, \qquad \wbar{c_1'}c_1' \equiv \wbar{c_2'}c_2' \equiv 1 \pmod{c_1 c_2}, \]
	then
	\[ S_{w_l}(\psi_m,\psi_n,(c_1c_1',c_2c_2')) = S_{w_l}(\psi_{m'},\psi_n,(c_1,c_2)) S_{w_l}(\psi_{m''},\psi_n,(c_1',c_2')), \]
	where $m' = \paren{\wbar{c_1'}^2 c_2' m_1, c_1' \wbar{c_2'}^2 m_2}$, and $m'' = \paren{\wbar{c_1}^2 c_2 m_1, c_1 \wbar{c_2}^2 m_2}$.
\end{lem}
Similarly, we have Weil-quality estimates for these sums, courtesy of Stevens \cite{Stevens}.
Keeping track of the dependence on the character there yields
\begin{thm}[Stevens]
	\begin{align*}
		& \abs{S_{w_l}(\psi_m,\psi_n,(A_1,A_2))}^2 \le \\
		& d(A_1)^2 d(A_2)^2 \paren{\abs{m_1 n_2}, D} \paren{\abs{m_2 n_1}, D} (A_1, A_2) A_1 A_2,
	\end{align*}
	where $D = \frac{A_1 A_2}{(A_1, A_2)}$.
\label{thm:Stevens}
\end{thm}
Dabrowski and Fisher \cite{DF01} have improved these estimates in most cases, but we expect that the exponents $(A_1 A_2)^{\frac{1}{2}}$ are sharp in the general case, though the author is unaware of any such proof.

The hope that these generalized Kloosterman sums will play a similar role to their classical counterparts leads us to make Linnik-type conjectures for cancellation between terms in a sum of $SL(3,\Z)$ Kloosterman sums, and the main result of this paper confirms this for a smooth weight function when the moduli are roughly the same size:
\begin{thm}
\label{thm:SumsOfKloostermanSums}
	Let $f:(\R^+)^2\to\C$ of compact support be eight-times differentiable in each variable, and take $X$ and $Y$ to be large parameters, with $\psi_m$ and $\psi_n$ non-degenerate characters, then
	\begin{align*}
		& \sum_{v\in\set{\pm 1}^2} \sum_{c_1,c_2 \ge 1} \frac{S_{w_l}(\psi_m,\psi_{v n},c)}{c_1 c_2} f\paren{X\frac{\pi c_2 \abs{m_1 n_2}}{c_1^2}, Y\frac{\pi c_1 \abs{m_2 n_1}}{c_2^2}} \\
		& \qquad \ll_{f,m,n,\epsilon} (XY)^{\epsilon}\paren{(XY)^{\frac{5}{14}}+X^{\frac{1}{2}} + Y^{\frac{1}{2}}}.
	\end{align*}
\end{thm}
If we instead apply Stevens' estimate for the individual Kloosterman sums, we are led to the bound $\paren{XY}^{\frac{1}{2} + \epsilon}$, so we are seeing cancellation between terms in the sum.
The $(XY)^{\frac{5}{14}}$ comes from the Kim-Sarnak bound on the Langlands parameters of $SL(3,\Z)$ cusp forms, and the $X^{\frac{1}{2}}$ and $Y^{\frac{1}{2}}$ terms come from some second-term asymptotics which present a difficulty in our partial inversion of a two-dimensional integral transform.
If the generalized Ramanujan-Selberg conjecture holds, then our bound becomes $(XY)^{\epsilon}\paren{X^{\frac{1}{2}} + Y^{\frac{1}{2}}}$, but we expect that the optimal bound would be $(XY)^{\epsilon}$ if one had a full inversion formula.

We expect that the most interesting examples should have $c_1 \asymp c_2$, i.e. when $X=Y$, and in this case the dominant term becomes $X^{\frac{5}{7}+\epsilon}$, which is entirely controlled by the Kim-Sarnak bound.
Again, under the generalized Ramanujan-Selberg conjecture, this becomes $X^{\frac{1}{2}+\epsilon}$ and the optimal bound should be $X^\epsilon$.

We have not chosen to track the dependence on the indices $m$ and $n$ here, but it is simple to do so.
The resulting bound is not close to optimal; essentially, we are multiplying the bound by powers of $m$ and $n$.
For comparison, Sarnak and Tsimerman \cite{SarnakTsimerman01} have made \thmref{thm:KuznetsovsBound} explicit in $m$ and $n$ with the bound
\[ \paren{x^{\frac{1}{6}} + (mn)^{\frac{1}{6}} + (m+n)^{\frac{1}{8}} (mn)^{\frac{7}{128}}} (mnx)^\epsilon, \]
and the third term may be removed if we assume the Ramanujan-Selberg Conjecture.
Similar bounds for the long-element Kloosterman sums on $SL(3)$ would require a great deal more work, and optimal bounds are not possible with the current method, again because of the error terms.

Finally, there is another new type of Kloosterman sum on $SL(3)$ which arises in the same manner, but is much smaller in summation.
There is some contention over whether this second type also has good cancellation in sums:
If it behaves as the examples we have studied so far, the answer should be yes, but Bump, Friedberg and Goldfeld have put forth a competing theory in \cite[Conjecture 1.2]{BFG01} to the effect that the Kloosterman zeta function of this sum should have poles on the boundary of its region of absolute convergence; in particular, this region would coincide with the region of conditional convergence, and there would be no significant cancellation between terms.

The methods here come from harmonic analysis on symmetric spaces.
Specifically, these results are obtained by studying a generalization of the Kuznetsov formula to $SL(3,\R)$:
Starting from a proof of Kuznetsov's trace formula on $SL(2,\R)$ by Zagier, and using the Fourier coefficient decomposition of automorphic forms on $SL(n,\R)$ by Friedberg, Li \cite{Goldfeld} has given a generalization of the first of Kuznetsov's trace formulas to $SL(n,\R)$ and this appears in Goldfeld's book on automorphic forms on $SL(n,\R)$ \cite{Goldfeld}.
So far, only the most basic of estimates have come out of the $SL(n)$ Kuznetsov formula and only for $SL(3)$, these may be found in a paper of Li herself \cite{Li01}, but in general, the integral transforms appearing in her formula are too complex to use effectively.
Blomer has been able to push somewhat farther by developing his own generalization of Kuznetsov's first formula \cite{Bl01}.

Using the Kuznetsov formula, we are able to express the integral transforms as an integral of the original test function against a function in Mellin-Barnes integral form.
With this representation, we can produce a sort of first-term inversion for the integral transform attached to the sum of long-element Kloosterman sums, which gives us a sort of incomplete generalization of Kuznetsov's second trace formula, and the proof of \thmref{thm:SumsOfKloostermanSums} then proceeds much as in Kuznetsov's original paper \cite{Kuz01}.

The central idea is that the spectral parameters of the $SL(3,\Z)$ automorphic forms occur in a strip which is positive distance from the region of absolute convergence of the long-element Kloosterman zeta function.
The aforementioned difficulties with the second-term asymptotics prevent us from obtaining the analytic continuation of the Kloosterman zeta function, but a similar path of shifting contours outside the region of absolute convergence yields the above results.

The paper \cite{BFG01} contains an alternate approach; they state, but do not explicitly prove, the meromorphic continuation of the \textit{unweighted} Kloosterman zeta function (the main object of study in the paper is weighted by a type of generalized Bessel function, much as the sum appearing in the spectral Kuznetsov formula), which would in principle give the above results without the error terms $X^{1/2}$ and $Y^{1/2}$, if one could control the growth of the Kloosterman zeta function on vertical lines in the complex plane.
On $SL(2)$, this method was started by Selberg \cite{Se02} (see \cite{Se03} as well as the G\"ottingen lecture in the second volume) and completed by Goldfeld and Sarnak \cite{GS01}.

Similarly, Yangbo Ye \cite{Y01} has given a third approach starting directly with sums of the long-element Kloosterman sums.
He provides a spectral interpretation which could be used to provide bounds in much the same manner as the current paper.
The difficulty with his Kuznetsov formula, as with Li's, lies in the complexity of the generalized Bessel functions, hence an analysis of the functions occurring in his formula, as we are about to provide for Li's, should produce similar results.
It would be an interesting problem for future research to compare the two.

\section{Background}
\label{sec:Background}
A good reference here is Goldfeld's book \cite{Goldfeld} for the automorphic forms side, but one should certainly start with the paper of Bump, Friedberg, and Goldfeld \cite{BFG01}.
For the harmonic analysis on symmetric spaces, the author learned from Terras' book \cite{T02}, but would like to recommend Jorgensen and Lang \cite{JL01}.
Before we begin, we have two notes on notation:
We tend to consider $SL(3,\R)$ embedded as the matrices of positive determinant in $GL(3,\R)/\cyclic{\R^+,\pm I}$, so when we discuss matrices of determinant other than one, we simply mean to divide by the cube root of the determinant.
Also, we parameterize by the Langland's parameters $\mu = (\mu_1,\mu_2,\mu_3)$ with $\mu_3 = -\mu_1-\mu_2$ in place of the parameters $\nu$ used by Goldfeld.

We write the Iwasawa decomposition for $z\in SL(3,\R)$ as $z \equiv xy \pmod{SO(3,\R)}$, where
\begin{align}
\label{eq:IwasawaXY}
	x \in U(\R) = \set{\Matrix{1&x_2&x_3\\&1&x_1\\&&1},x_i \in \R}, \\
\label{eq:IwasawaY}
	y \in Y(\R)= \set{\Matrix{y_1 y_2 \\ & y_1 \\ &&1},y_i \in \R^+};
\end{align}
note the ordering of the indices.
The $SL(3,\R)$-invariant measure has the form
\begin{align}
\label{eq:IwasawaHaar}
	dz = dx\,dy,\qquad dx = dx_1\,dx_2\,dx_3, \qquad dy = \frac{dy_1\,dy_2}{(y_1 y_2)^3},
\end{align}
and we define characters on the space of such $x$ and $y$ as \\ $\psi_{(m_1,m_2)}(x)=\e{m_1 x_1 + m_2 x_2}$, and $p_\mu(y) = y_1^{\mu_1+\mu_2} y_2^{\mu_1}$.
The power function is normalized by $\rho = (1,0,-1)$ so that $p_{\rho+\mu}(xy) = y_1^{1+\mu_1+\mu_2} y_2^{1+\mu_1}$.
As another notational quirk, we tend not to differentiate between the matrix $y$ and the pair $y = (y_1, y_2)$, and for signed pairs $m$, we will tend to write $\abs{m} = (\abs{m_1}, \abs{m_2})$.
Lastly, when integrating over the space $Y(\R)$ we use the measure $dy$ above.

In terms of the Langlands parameters, the definition of $SL(3,\Z)$ Maass cusp forms becomes: Writing $G=SL(3,\R)$, $K=SO(3,\R)$, and $\Gamma=SL(3,\Z)$, $\varphi:G/K\to\C$ is a Maass cusp form if $\varphi(\gamma z) = \varphi(z)$ for all $\gamma\in \Gamma$, $\varphi \in L^2(\Gamma\backslash G/K)$, 
\[ \int_{U_i^*(\Z)\backslash U_i^*(\R)} \varphi(uz) du = 0, \]
for the upper-triangular groups
\[ U_1^* = \set{\Matrix{1&0&*\\&1&*\\&&1}}, \qquad U_2^* = \set{\Matrix{1&*&*\\&1&0\\&&1}}, \]
and $\varphi$ is an eigenvalue of all of $\mathfrak{D}$, the group of $G$-invariant differential operators on $G/K$.
A Maass form is of type $\mu = (\mu_1, \mu_2, \mu_3)$ (where $\mu_1 + \mu_2 + \mu_3 = 0$) if it shares the eigenvalues of the power function at $\mu$:
\[ -\Delta_1 p_{\rho+\mu} = \paren{1-\frac{\mu_1^2+\mu_2^2+\mu_3^2}{2}} p_{\rho+\mu},\qquad \Delta_2 p_{\rho+\mu} = -\mu_1 \mu_2 \mu_3 p_{\rho+\mu}, \]
where $\Delta_1$ and $\Delta_2$ are the explicit generators of $\mathfrak{D}$, given in \cite[p153]{Goldfeld}.
It is known that $-\wbar{\mu}$ is some permutation of $\mu$ for the Langlands parameters of $SL(3,\Z)$ Maass cusp forms \cite{Va01}, and we will also need a result of Kim and Sarnak that $\abs{\Re(\mu_i)} \le \frac{5}{14}$ \cite{LRS01,LRS02,KS01}.

For the Langlands Eisenstein series, we have two types:
\[ E(z,\mu) = \sum_{\gamma \in \Gamma_\infty\backslash \Gamma} p_{\rho+\mu}(\gamma z), \qquad E_\phi(z,\mu_1) = \sum_{\gamma \in P_{2,1} \backslash \Gamma} p_{\mu_1,\phi}(\gamma z), \]
where
\[ \Gamma_\infty = \set{\Matrix{1&*&*\\0&1&*\\0&0&1} \in \Gamma}, \qquad P_{2,1} = \set{\Matrix{*&*&*\\ *&*&*\\0&0&1} \in \Gamma}, \]
and $p_{\mu_1,\phi}(g) = (y_1^2 y_2)^{\frac{1}{2}+\mu_1} \phi(x_2+iy_2)$ with $\phi$ any $SL(2,\Z)$ cusp form.
The Eisenstein series are smooth functions on $\Gamma\backslash G/K$ which are eigenvalues of all of $\mathfrak{D}$, but not square-integrable.
They have meromorphic continuation to all of $\C$ in each $\mu_i$.

We will be making heavy use of the Jacquet-Whittaker function, which is initially defined by
\[ W(z;\mu,\psi_m) = \int_{U(\R)} p_{\rho+\mu}(w_l uz) \psi_m(u) du, \]
where $U(\cdot) \subset SL(3,\cdot)$ to be the Borel subgroup, i.e. the group of upper triangular matrices with ones on the diagonal.
The completed Whittaker function is given by $W^*(z;\mu,\psi_m) = \Lambda(\mu) W(z;\mu,\psi_m)$, where
\[ \Lambda(\mu) = \pi^{-\frac{3}{2}+\mu_3-\mu_1} \Gamma\paren{\frac{1+\mu_1-\mu_2}{2}} \Gamma\paren{\frac{1+\mu_1-\mu_3}{2}} \Gamma\paren{\frac{1+\mu_2-\mu_3}{2}}. \]

We will need a number of basic facts about the Whittaker function, so we collect them here:
The dependence on $x$ is is given by $W(xy;\mu,\psi_m) = \psi_m(x) W(y;\mu,\psi_m)$.
It is also easy to check that if $0 \ne t_1,t_2 \in \R$,
\[ W(y;\mu,\psi_{t_1 t_2}) = p_{-\rho-\mu^{w_l}}(\abs{t}) W(\abs{t} y;\mu,\psi_{11}), \]
by sending $u\mapsto t^{-1} v u t v$, where
\[ \abs{t} = \Matrix{\abs{t_1} \abs{t_2} \\ & \abs{t_1} \\&&1}, \qquad v = \Matrix{\sgn(t_1)\\&\sgn(t_1 t_2)\\&&\sgn(t_2)}, \]
and $\mu^{w_l} = (\mu_3,\mu_2,\mu_1)$.
Thus we only need an analytic expression for $W(y,\mu,\psi_{11})$.

We have the double Mellin transform pair \cite[6.1.4, 6.1.5]{Goldfeld}
\begin{align}
\label{eq:WhittakerMellinExpand}
	W^*(y,\mu,\psi_{11}) &= -\frac{1}{16\pi^4} \int_{\Re(u) = (2,2)} G(u, \mu) (\pi y_1)^{1-u_1} (\pi y_2)^{1-u_2} du, \\
	G(u, \mu) &= \frac{4}{\pi^2} \int_{Y(\R)} W^*(y,\mu,\psi_{11}) (\pi y_1)^{1+u_1} (\pi y_2)^{1+u_2} dy, \nonumber
\end{align}
where
\[ G(u, \mu) = \frac{\Gamma\paren{\frac{u_1-\mu_1}{2}}\Gamma\paren{\frac{u_1-\mu_2}{2}}\Gamma\paren{\frac{u_1-\mu_3}{2}}\Gamma\paren{\frac{u_2+\mu_1}{2}}\Gamma\paren{\frac{u_2+\mu_2}{2}}\Gamma\paren{\frac{u_2+\mu_3}{2}}}{\Gamma\paren{\frac{u_1+u_2}{2}}}. \]
These give the asymptotic $W^*(y,\mu,\psi_{11}) \ll_\mu y_1^{1-t_1} y_2^{1-t_2}$ for any \\ $t_1 \ge c_1 = \max_i\set{\Re(-\mu_i)}$, $t_2 \ge c_2 = \max_i\set{\Re(\mu_i)}$.
In particular, the Mellin transform of the Whittaker function converges absolutely for $\Re(t_1) > c_1$,$\Re(t_2) > c_2$.

Using the Jacquet-Whittaker function, our particular choice of normalizations for the Fourier-Whittaker coefficents of a Maass cusp form $\varphi$, for example, are given by
\begin{align}
\label{eq:FourierWhittakerNormalization}
	\int_{U(\Z)\backslash U(\R)} \varphi(uy) \psi_m(u) du = \frac{\rho_\varphi(m)}{\abs{m_1 m_2}} W^*(\abs{m}y;\mu,\psi_{11}),
\end{align}
where by $\abs{m}$ we mean the pair $(\abs{m}_1, \abs{m}_2)$.
We should point out that the above normalization controls the growth of the Fourier-Whittaker coefficients, but says nothing about the normalization of the function $\varphi$ itself.
We use orthonormalization, as stated in \thmref{thm:LisKuznetsovFormula}, but these formulae are frequently given in Hecke-normalization, so that
\[ \rho_\varphi^\text{Hecke}(m) = \frac{\rho_\varphi(m)}{\rho_\varphi(1,1)}, \]
the conversion then involves a residue of the Rankin-Selberg $L$ function associated to $\varphi$ \cite{Bl01}:
\begin{thm}
	For $\rho_\varphi$, $\eta_\phi$, and $\eta$ the Fourier-Whittaker coefficients of $\varphi$, $E_\phi$, and $E$, respectively, we have
	\begin{align*}
		\frac{\rho_\varphi(m)}{\abs{\Lambda(\mu)}} =& \frac{2}{\sqrt{3}} \frac{\rho_\varphi^\text{Hecke}(m)}{\sqrt{\res_{s=1} L(\varphi \times \wtilde{\varphi},s)}}, \\
		\frac{\eta_\phi(m;\mu_1)}{\Lambda(\mu)} =& C_1 \frac{\eta_\phi^\text{Hecke}(m;\mu_1)}{L(\phi, 1+\mu) \sqrt{L(\text{sym}^2 \phi,1)}}, \text{ for } \Re(\mu_1) = 0, \\
		\frac{\eta(m;\mu)}{\Lambda(\mu)} =& C_2 \frac{\eta^\text{Hecke}(m;\mu)}{\zeta(1+\mu_1-\mu_2)\zeta(1+\mu_1-\mu_3)\zeta(1+\mu_2-\mu_3)}, \text{ for } \Re(\mu) = 0,
	\end{align*}
	for some absolute constants $C_1$, $C_2$.
\end{thm}
This varies somewhat from \cite{Bl01} due to slightly different normalizations.

We will require both the abstract and explicit definitions of the $SL(3,\Z)$ Kloosterman sums, so we start by defining
\begin{align}
\label{eq:DegenerateU}
	U_w=(w^{-1}\, U \, w)\cap U, \qquad \text{and} \qquad \wbar{U}_w=(w^{-1}\, \trans{U} \, w)\cap U,
\end{align}
where $w$ is an element of the Weyl group $W$, the Bruhat decomposition of an element $g\in SL(3,\R)$ is $g=b_1 cw b_2$ with $b_1,b_2\in U(\R)$ (the decompostion is unique if we require $b_2 \in \wbar{U}_w(\R)$), $w\in W$, and
\[ c = \Matrix{\frac{1}{c_2}&&\\&\frac{c_2}{c_1}&\\&&c_1}, c_1,c_2\in\N. \]
The Weyl group $W$ has the 6 elements
\begin{equation*}
	\begin{array}{rlrlrlrl}
		I &= \Matrix{1\\&1\\&&1}, & w_2 &= \Matrix{&1\\-1\\&&1}, & w_3 &= \Matrix{1\\&&-1\\&1}, \\
		w_4 &= \Matrix{&1\\&&1\\1}, & w_5 &= \Matrix{&&1\\1\\&1}, & w_l &= \Matrix{&&1\\&-1\\1}.
	\end{array}
\end{equation*}
The generalized Kloosterman sums we study here are exponential sums attached to the Bruhat decomposition:
We set
\[ S_w(\psi_m,\psi_n, c) = \sum_{b_1 cw b_2 \in U(\Z)\backslash \Gamma / V\wbar{U}_w(\Z)} \psi_m(b_1) \psi_n(b_2), \]
matrices provided the sum is independent of the choice of Bruhat decompositions of each coset representative, and 0 otherwise.
This independence assumption is called the compatibility condition, and is equivalent to the condition
\begin{align}
\label{eq:CompatibilityCondition}
	\psi_m((cvw) u (cvw)^{-1}) \psi_n(u^{-1}) = 1, \forall u\in U_w(\R).
\end{align}
Here $V$ is the group of diagonal orthgonal matrices:
\[ V=\set{I, \quad \Matrix{-1\\&-1\\&&1}, \quad \Matrix{-1\\&1\\&&-1}, \quad \Matrix{1\\&-1\\&&-1}}. \]

We wish to explicitly define the Kloosterman sums, which is best done in terms of the Pl\"ucker coordinates:
Coset representatives
\[ \Matrix{*&*&*\\d&e&f\\a&b&c} \in U(\Z)\backslash\Gamma \]
are characterized by six invariants:
The bottom row $A_1 = a, B_1 = b, C_1 = c$ having $(A_1, B_1, C_1) = 1$, and the first set of minors $A_2 = bd-ae$, $B_2 = af-cd$, $C_2 = ce - bf$ having $(A_2, B_2, C_2) = 1$ and subject to $A_1 C_2 + B_1 B_2 + C_1 A_2 = 0$.
In \cite{BFG01}, the Bruhat decomposition for each element of the Weyl group were computed using these invariants -- note that membership in a particular Bruhat cell imposes certain requirements on the Pl\"ucker coordinates, giving explicit forms to the $SL(3,\Z)$ Kloosterman sums.
The three degenerate sums are listed in \tableref{tab:DegenerateKloostermanSums}, using the classical Kloosterman sum.
\begin{table}
\begin{tabular}{l|lll}
	& Bruhat & Compatibility & $S_w(\psi_m, \psi_n, (c_1,c_2))$ \\
	\hline
	$I$ & $c_1=c_2=1$ & $m=n$ & 1 \\
	$w_2$ & $c_1 = 1$ & $m_1=n_1=0$ & $S(-m_2, -n_2, c_2)$ \\
	$w_3$ & $c_2 = 1$ & $m_2=n_2=0$ & $S(m_1, n_1, c_1)$
\end{tabular}
\caption{Degenerate $SL(3,\Z)$ Kloosterman Sums}
\label{tab:DegenerateKloostermanSums}
\end{table}

The $w_4$ Kloosterman sum is a new exponential sum.
Its Bruhat condition is $c_2 | c_1$ and the compatibility condition is
\begin{align}
\label{eq:w4Compatibility}
	m_2 c_1 = n_1 c_2^2.
\end{align}
Explicitly,
\begin{align*}
	& S_{w_4}(\psi_m, \psi_n, (A_1, B_2)) = \\
	& \sum_{\substack{C_2 \pmod{B_2}\\ C_1 \pmod{A_1}\\(A_1/B_2, C_1) = (B_2, C_2) = 1}} \e{-m_2 \frac{\wbar{C_2} C_1}{B_2}-m_1\frac{\wbar{C_1} B_2}{A_1}-n_2\frac{C_2}{B_2}}.
\end{align*}

The $w_5$ Kloosterman sum is essentially the same as for $w_4$.
Its Bruhat condition is $c_1 | c_2$ and the compatibility condition is
\begin{align}
\label{eq:w5Compatibility}
	m_1 c_2 = n_2 c_1^2,
\end{align}
and we have
\[ S_{w_5}(\psi_m, \psi_n, (c_1,c_2)) = S_{w_4}(\psi_{-m_2,m_1}, \psi_{n_2,-n_1}, (c_2,c_1)). \]

There are Weil-quality bounds for this first type of Kloosterman sum due to Larsen (in \cite{BFG01}):
\begin{thm}[Larsen] \ 

\begin{enumerate}
	\item
		\begin{align*}
			& \abs{S_{w_4}(\psi_m,\psi_n,c)} \le \\
			& \Min{d(c_2)^\varkappa \paren{\abs{m_1}, \frac{c_1}{c_2}} c_2^2, d(c_1) (\abs{m_1}, \abs{n_2}, c_2) c_1},
		\end{align*}
	\item
		\begin{align*}
			& \abs{S_{w_5}(\psi_m,\psi_n,c)} \le \\
			& \Min{d(c_1)^\varkappa \paren{\abs{m_2}, \frac{c_2}{c_1}} c_1^2, d(c_2) (\abs{m_2}, \abs{n_1}, c_1) c_2},
		\end{align*}
\end{enumerate}
	where $\varkappa=\frac{\log 3}{\log 2}$.
\label{thm:Larsen}
\end{thm}

The Bruhat and compatibility conditions for the long-element Kloosterman sum are vacuously true, and we note that in \thmref{thm:Stevens}, Stevens is completely unconcerned with the dependence of his estimate on the indices $m$ and $n$, so one needs to keep track of this:
In his proof of Theorem (5.9), on page 49, use instead the estimates
\begin{align*}
	\paren{\abs{\nu_1 p^{s-a}}_p^{-1}, \abs{\nu_2' p^{r-b}}_p^{-1}, p^r} \le& \paren{\abs{\nu_1 \nu_2'}_p^{-1},p^r} \paren{p^{s-a}, p^{r-b}} \\
	\le& \paren{\abs{\nu_1 \nu_2'}_p^{-1},p^r} p^{\frac{s-a+r-b}{2}},
\end{align*}
and similarly
\[ \paren{\abs{\nu_2 p^{2r-s-b}}_p^{-1}, \abs{\nu_1' p^{r-a}}_p^{-1}, p^r} \le \paren{\abs{\nu_2 \nu_1'}_p^{-1},p^r} p^{\frac{2r-s-b+r-a}{2}}. \]

The Kuznetsov trace formula relates sums of Fourier coefficients of $SL(2,\Z)$ automorphic forms to sums of the classical Kloosterman sums.
It has two forms:  The first allows an essentially arbitrary test function on the Fourier coefficient side and has transforms of the test function on the Kloosterman sum side; the second form has the test function on the Kloosterman sum side with transforms of it on the Fourier coefficient side.
This is a type of asymmetrical Poisson summation formula, but having both forms allows us to use it in much the same manner -- to study sums of Kloosterman sums using our knowledge of sums of Fourier coefficents of automorphic forms and visa versa.

Starting from a proof of Kuznetsov's trace formula on $SL(2,\R)$ by Zagier, Li has given a generalization of the first form to $SL(n,\R)$; in the case of $SL(3,\R)$, it becomes:
\begin{thm}[Li]
\label{thm:LisKuznetsovFormula}
	Let $\set{\varphi}$ be an orthonormal basis of the $SL(3,\Z)$ cusp forms with Langlands parameters $\mu_\varphi$, and $\set{\phi}$ an orthonormal basis of $SL(2,\Z)$ cusp forms with Langlands parameters $\mu_\phi$.
	Let $k \in C_c^\infty(K\backslash G/K)$, and $m, n$ pairs of non-zero integers.
	Then
	\begin{align}
		\label{eq:LisKuznetsovFormula}
		& \sum_{\varphi} \frac{\hat{k}(\mu_\varphi)}{C(\mu_\varphi)} \rho_\varphi(n) \wbar{\rho_\varphi}(m) \\
		& \qquad +\frac{1}{4\pi i} \sum_{\phi} \int_{\Re(\mu_1)=0} \frac{\hat{k}\paren{\mu_1-\mu_\phi,-2\mu_1}}{C\paren{\mu_1-\mu_\phi,-2\mu_1}} \eta_\phi(n; \mu_1) \wbar{\eta_\phi}(m; \mu_1) d\mu_1 \nonumber \\
		& \qquad +\frac{1}{(2\pi i)^2} \int_{\Re(\mu)=(0,0)} \frac{\hat{k}(\mu)}{C(\mu)} \eta(n; \mu) \wbar{\eta}(m; \mu) d\mu \nonumber \\
		&= \delta_{\abs{m}=\abs{n}} H_I(k, \psi_m,\psi_n,c) \nonumber \\
		& \qquad +\sum_{w\in\set{w_4,w_5,w_l}}\sum_{v\in V} \sum_{c_1,c_2\in\N} S_w(\psi_m,\psi_n^v,c) H_w(k, \psi_m,\psi_n^v,c), \nonumber
	\end{align}
	where
	\[ C(\mu) = \cos \frac{\pi}{2} (\mu_1-\mu_2) \cos \frac{\pi}{2} (\mu_1-\mu_3) \cos \frac{\pi}{2} (\mu_2-\mu_3), \]
	$\rho_\varphi$, $\eta_\phi$, and $\eta$ are the Fourier-Whittaker coefficients of $\varphi$, $E_\phi$, and $E$, respectively, normalized as in \eqref{eq:FourierWhittakerNormalization}, $\hat{k}$ is the generalized Selberg transform of $k$
	\begin{align}
	\label{eq:SphericalTransform}
		\hat{k}(\mu) = \int_{G} k(z) \wbar{p_{\rho+\mu}(z)} dz,
	\end{align}
	and $H_w(k, \psi_m,\psi_n^v,c)$ is given by the integral
	\begin{align}
		& \frac{2\abs{m_1 m_2 n_1 n_2}}{\pi} \int_{Y(\R)} \int_{U(\R)} \int_{\wbar{U}_w(\R)} k\paren{\abs{m}\paren{x t}^{-1} cw(x't)\abs{n}^{-1}} \\
		& \qquad \psi_m(x) \wbar{\psi_n^v(x')} dx'\,dx\,t_1^2 t_2\,dt, \nonumber
	% \frac{2^{n-1} \Gamma\paren{\frac{n}{2}}}{\pi^{n/2}}=\frac{2}{\pi} %
	\end{align}
	with $\psi_n^v(x) = \psi_n(vxv)$, and $U(\R)$, $Y(\R)$, $\wbar{U}_w(\R)$ and $dt$ given by eqs. \eqref{eq:IwasawaXY}, \eqref{eq:DegenerateU} and \eqref{eq:IwasawaHaar}.
\end{thm}
The proof of this is given in Goldfeld's book, with two corrections:
The final formula for $H_w$ should have $U(\R)$ in place of $U(\Z)\backslash U(\R)$ and $cw$ in place of $wc$.

Note that we have defined our Kloosterman sums $S_w$ to be zero if the compatibility condition \eqref{eq:CompatibilityCondition} is not met.
This is possible because only the well-defined terms occur in the Fourier coefficients of Poincar\'e series (see \cite[Table (5.4)]{BFG01} for $SL(3)$, \cite[pp173-4]{Friedberg} for $SL(n)$), and hence also in Kuznetsov formula.
One may verify from the compatibility conditions listed in \eqref{eq:w4Compatibility}, \eqref{eq:w5Compatibility}, and \tableref{tab:DegenerateKloostermanSums} -- vacuous for the long element -- that the trivial element $I$, long element $w_l$ and the two intermediate elements $w_4$ and $w_5$ are the only Weyl elements which appear for non-degenerate characters on $SL(3,\Z)$.

The test function $k$ appears on the Fourier coefficient side of the formula as the Selberg transform $\hat{k}$, but this has an inversion formula, called spherical inversion, which allows us to replace $\hat{k}$ with an essentially arbitrary test function, and so this formula generalizes the first form of Kuznetsov's formula on $SL(2,\R)$.

We define two spaces:
The first space is the Harish-Chandra Schwartz space $\mathrm{HCS}(K\backslash G/K)$ of functions $f:G\to\C$ which are bi-$K$-invariant, smooth in the coordinates of $G$, and for each bi-$K$-invariant differential operator $D$ on $G$ and any $N\in\N$, we have
\[ \abs{f(a)} \ll_{D,N} \frac{h_0(a)}{(1+\abs{\log p_{11}(a)})^N}, \]
for all positive diagonal $a$.
The second space is the Schwartz space $\mathrm{SCH}^W(\mu)$ of real-analytic functions on $\Re(\mu) = 0$ which are invariant under the action of the Weyl group (which acts by permutations of the coordinates of $\mu$).
Then we may state the spherical inversion formula for $SL(3,\R)$ as:
\begin{thm}[Spherical Inversion]
	The Selberg transform $k \mapsto \hat{k}$ in \eqref{eq:SphericalTransform} extends to an isomorphism
	\[ \mathrm{HCS}(K\backslash G/K) \stackrel{\sim}{\to} \mathrm{SCH}^W(\mu), \]
	with inverse given by
	\[ k(z) = -\frac{1}{64\pi^4} \int_{\Re(\mu)=(0,0)} \frac{\hat{k}(\mu)}{\abs{c_3(\mu)}^2} h_\mu(z) d\mu, \]
	where
	\begin{align}
	\label{eq:c3tandef}
		\abs{c_3(\mu)}^{-2} = \frac{4}{\pi^2} \prod_{1\le i< j\le 3} -\frac{\pi}{2}(\mu_i-\mu_j)\tan\frac{\pi}{2}(\mu_i-\mu_j),
	\end{align}
	and
	\[ h_\mu(z) = \int_K p_{\rho+\mu}(kz) dk, \quad \int_K dk = 1, \]
	is the spherical function.
\end{thm}
The notation $\abs{c_3(\mu)}^{-2}$ may be somewhat confusing, but this function was originally defined as the squared modulus of a product of beta functions.
The equality \eqref{eq:c3tandef} technically only holds on the line $\Re(\mu) = (0,0)$, but in practice when we use this function, we mean the holomorphic function of $\mu$ on the right.
The most explicit reference here is \cite[p. 100 eq. 3.23]{T02}, and the leading constant is just
\[ \frac{n}{2^{n-1}} (2\pi i) \omega_n = \frac{\prod_{j=1}^n \Gamma\paren{\frac{j}{2}}}{(2\pi i)^{n-1} (n-1)! \, \pi^{\frac{n(n+1)}{4}}}, \]
as the extra $b_n$ appearing in \cite{T02} comes from decomposing $G=KAK$, with $A$ the diagonal matrices, and then choosing a Weyl chamber.
The extra $\frac{n}{2^{n-1}}$ comes from the $s \mapsto \mu$ substitution: $2s_i = (\rho_i+\mu_i)-(\rho_{i+1}+\mu_{i+1})$.

To apply the Kuznetsov formula, we will need to know when the various sums and integrals converge.
From the Weil-quality bounds on the Kloosterman sums -- \thmref{thm:Stevens} and \thmref{thm:Larsen} -- we may investigate the absolute convergence of the corresponding Kloosterman zeta functions:
\begin{cor}
\begin{enumerate}
	\item \[ \sum_{c_1, c_2 \in \N} \frac{\abs{S_{w_4}(\psi_m,\psi_n,c))}}{c_1 c_2} c_2^{3 u} \]
		converges on $u < 0$,
	\item \[\sum_{c_1, c_2 \in \N} \frac{\abs{S_{w_5}(\psi_m,\psi_n,c))}}{c_1 c_2} c_1^{3 u} \]
		converges on $u < 0$, and
	\item \[ \sum_{c_1, c_2 \in \N} \frac{\abs{S_{w_l}(\psi_m,\psi_n,c)}}{c_1 c_2} \paren{\frac{c_1^2}{c_2}}^{u_1} \paren{\frac{c_2^2}{c_1}}^{u_2} \]
		converges on $2u_1-u_2 < -\frac{1}{2}$, $-u_1+2u_2 < -\frac{1}{2}$.
\end{enumerate}
\label{cor:AbsConvofKloostermanZetaFuncs}
\end{cor}
These normalizations are unusual, but will make more sense once we start evaluating the integral transforms.

A simple consequence of the Kuznetsov formula are some mean value estimates for Fourier-Whittaker coefficients of $SL(3,\Z)$ automorphic forms, which will help us evaluate convergence of the spectral side of the Kuznetsov formula:
\begin{thm}[Blomer]
\label{thm:MeanValueEstimates}
	For $\mu^\dagger$ and $T \ge 1$ fixed, the quantities
	\[ \sum_{\norm{\mu_\varphi-\mu^\dagger}\le T} \frac{\abs{\rho_{\varphi}(1,1)}^2}{C(\mu_\varphi)}, \]
	\[ \sum_{\phi} \underset{\substack{\Re(\mu_1)=0 \\ \norm{\paren{\mu_1-\mu_\phi,-2\mu_1}-\mu^\dagger} \le T}}{\int} \frac{\abs{\eta_\phi((1,1); \mu_1)}^2}{C\paren{\mu_1-\mu_\phi,-2\mu_1}} d\mu_1, \]
	and
	\[ \underset{\substack{\Re(\mu)=(0,0) \\ \norm{\mu-\mu^\dagger}\le T}}{\int} \frac{\abs{\eta((1,1); \mu)}^2}{C(\mu)} d\mu \]
	are all bounded by
	\[ T^2 \paren{T+\norm{\mu^\dagger}}^3. \]
\end{thm}
This is essentially a theorem of Blomer, but one can obtain these by applying \thmref{thm:KuznetsovSimplification} with a test function $\hat{k}$ which is non-negative on the spectrum and decays rapidly away from the desired regions then applying bounds for the $J_{w,\mu}$ functions.
Such test functions are constructed in \cite{DKV01} -- see \cite{Li01}, and we will prove bounds of the requisite nature in \propref{prop:JwHoloAndBd}.
Blomer demonstrates results of this type by applying a Kuznetsov formula of his own, the purpose of which is the same; that is, to make a sufficiently simplified version of the Kuznetsov formula on $SL(3,\R)$.
One can extend this to the $m \ne (1,1)$ Fourier coefficents by applying the second half of the Kim-Sarnak result, $\frac{\rho_\phi(m)}{\rho_\phi(1,1)} \ll (m_1 m_2)^{\frac{5}{14}+\epsilon}$, or by applying the Kuznetsov formula directly, which results in a slightly different bound.

\section{Methods}
\label{sec:Methods}
We start with Li's generalization of the Kuznetsov formula, whose complexity leads us to our first technical theorem.
To that end, we give some preliminary definitions:
First, set $G^*(u,\mu) = \frac{G(u,\mu)}{\Lambda(\mu)}$, and let $G^*_l(u_2,\mu)$ be its residue at $u_1 = \mu_1$, $G^*_r(u_1,\mu)$ the residue at $u_2=-\mu_2$, and $G^*_b(\mu)$ the double residue.
For $w\in W$, $x' \in U(\R)$, and $t \in Y(\R)$, set $x^* y^* \equiv w x' \pmod{K}$, and also $t^w = w t w^{-1}$.
Then we define an action of $w$ on exponents $u$ with $u^w$ defined by $t_1^{u^w_1} t_2^{u^w_2} = (t^w_1)^{u_1} (t^w_2)^{u_2}$.
In particular, $w$ acts by permutations on $\mu$ as in the usual definition $p_{\mu^w}(t) = p_\mu(t^w) = p_\mu(w t)$.

Fix $\Delta > 0$, and for $y\in (\R\setminus\set{0})^2$, $w=w_4,w_5,w_l$ define the functions
\begin{align}
\label{eq:JwDef}
	J_{w,\mu}(y) =& \frac{1}{16 \pi^3} \frac{k_\text{adj}(\mu)}{\abs{c_3(\mu)}^2} \Biggl(\frac{-1}{16 \pi^4} \int_{\Re(u) = -\frac{1}{2}-10\epsilon} G^*(u,\mu) T_w(u,y) du \\
	& +\frac{-3i}{8 \pi^3} \int_{\Re(u_2) = -\frac{1}{2}-10\epsilon} G^*_l(u_2,\mu) T_w((\mu_1,u_2),y) du_2 \nonumber\\
	& +\frac{-3i}{8 \pi^3} \int_{\Re(u_1) = -\frac{1}{2}-10\epsilon} G^*_r(u_1,\mu) T_w((u_1,-\mu_2),y) du_1 \nonumber\\
	& +\frac{6}{4 \pi^2} G^*_b(\mu) T_w((\mu_1,-\mu_2),y) \Biggr), \nonumber
	% -1/16 pi^7 from previous and (2\pi i)^2 c_1 c_2 from H_w %
\end{align}
\begin{align}
\label{eq:TwFinal}
	T_w(u,y) =& (\pi\abs{y_1})^{-u_1} (\pi\abs{y_2})^{-u_2} \int_{Y(\R)} W(t,-\mu,\psi_{11}) X'_w(u, y, t) \\
	& \qquad t_1^{3-u^w_1+2\Delta} t_2^{2-u^w_2+\Delta} \, dt, \nonumber
\end{align}
\begin{align*}
	X'_w(u,y,t) =& \int_{\wbar{U}_w(\R)} \psi_{y t^w}(x^*)\wbar{\psi_{t}}(x') {y_1^*}^{1-u_1} {y_2^*}^{1-u_2} dx'.
\end{align*}
Here the function $\abs{c_3(\mu)}^2$ is the holomorphic function defined by \eqref{eq:c3tandef}, $Y(\R)$ and $\wbar{U}_w(\R)$ are defined by \eqref{eq:IwasawaY} and \eqref{eq:DegenerateU}, and
\begin{align}
\label{eq:kadjdef}
	k_\text{adj}(\mu) = \prod_{j<k} \paren{(3+2\Delta)^2-\paren{\mu_j-\mu_k}^2}^{-\Delta/2},
\end{align}
is a minor correction factor, whose existence will be justified later; to be precise, we take the power function to be real on the real axis with the branch outside the strip $\abs{\Re(\mu_i)} < 1+\Delta$.
For clarity, we will explicitly compute most of the abstractly defined quantities given here in section \ref{sec:KSimpWeylTerms}.

For the identity Weyl element, a.k.a. the trivial term, we set
\begin{align}
\label{eq:JIDef}
	J_I(\mu) = \frac{\prod_{j<k} -\frac{\pi}{2}(\mu_j-\mu_k) \sin \frac{\pi}{2}(\mu_j-\mu_k)}{C^*(\mu)},
\end{align}
where
\[ C^*(\mu) = \pi^{3/2} \Gamma\paren{\frac{1+\Delta}{2}}^{-3} \prod_{j<k} \frac{\paren{(3+2\Delta)^2-\paren{\mu_j-\mu_k}^2}^{\Delta/2}}{\Gamma\paren{\frac{1+\Delta+\mu_j-\mu_k}{2}}\Gamma\paren{\frac{1+\Delta+\mu_k-\mu_j}{2}}}. \]

With the functions defined above, in section \ref{sec:IntegralTransforms} we prove:
\begin{thm}
\label{thm:KuznetsovSimplification}
	Fix $\Delta > 0$ and let $\hat{k}(\mu)$ be symmetric in $(\mu_1,\mu_2,\mu_3)$, $\mu_3=-\mu_1-\mu_2$, holomorphic in each variable on $\Re(\mu_1,\mu_2) = \eta \in \braces{-\frac{1}{2} - \Delta, \frac{1}{2}+\Delta}^2$, of sufficient decay that the integral
	\begin{align}
		\label{eq:KuzSimplkBd}
		& \int_{\Re(\mu)=\eta} \abs{\hat{k}(\mu)} \abs{\mu_1-\mu_2}^{\frac{27+4\Delta}{16}+\epsilon} \abs{\mu_1-\mu_3}^{\frac{27+4\Delta}{16}+\epsilon} \abs{\mu_2-\mu_3}^{\frac{1-4\Delta}{8}+\epsilon} \abs{d\mu},
	\end{align}
	converges, then we have the formula \eqref{eq:LisKuznetsovFormula}, where now
	\begin{align*}
		H_I(k,\psi_m,\psi_n,c) =& -\frac{1}{32 \pi^8} \int_{\Re(\mu)=(0,0)} \hat{k}(\mu) J_I(\mu) d\mu, \\
		H_{w_4}(k,\psi_m,\psi_n,c) =& \frac{1}{(2 \pi i)^2 c_1 c_2} \int_{\Re(\mu)=\eta} \hat{k}(\mu) J_{w_4,\mu}\paren{\frac{m_1 m_2^2 n_2}{c_2^3 n_1}} \, d\mu, \\
		H_{w_5}(k,\psi_m,\psi_n,c) =& \frac{1}{(2 \pi i)^2 c_1 c_2} \int_{\Re(\mu)=\eta} \hat{k}(\mu) J_{w_5,\mu}\paren{\frac{m_1^2 m_2 n_1}{c_1^3 n_2}} \, d\mu, \\
		H_{w_l}(k,\psi_m,\psi_n,c) =& \frac{1}{(2 \pi i)^2 c_1 c_2} \int_{\Re(\mu)=\eta} \hat{k}(\mu) J_{w_l,\mu}\paren{\frac{c_2 m_1 n_2}{c_1^2}, \frac{c_1 m_2 n_1}{c_2^2}} \, d\mu,
	\end{align*}
	with $J_{w,\mu}$ given by the integrals \eqref{eq:JwDef}, \eqref{eq:TwFinal},\eqref{eq:XwlFinal}, \eqref{eq:Xw4Final}, and \eqref{eq:Xw5Final}, $J_I(\mu)$ given by \eqref{eq:JIDef}, and $C(\mu)$ replaced with $C^*(\mu)$.
\end{thm}
This should be regarded as a theorem on the higher-rank hypergeometric functions, in the style of Stade; it assigns to the weight functions $H_w$ good complex analytic expressions.
Lastly, we needed to replace $C(\mu)$ because the pole at $\mu_1-\mu_2=-1$ also shows up on the arithmetic side and interfers with the absolute convergence of the sum of the long-element Kloosterman sums; we accomplish this change by applying Stade's formula at $s=1+\Delta$ instead of $s=1$, and renormalizing to obtain the proper asymptotics, $C(\mu) \asymp C^*(\mu)$.
All of the $J_{w,\mu}$ functions, as well as $C^*(\mu)$, depend on the choice of $\Delta$.

Acting on the desire for a useable Kuznetsov formula on $SL(3)$, we study the weight functions $J_{w,\mu}$.
We express asymmetrical polynomial bounds in the $\mu$ variables through
\begin{align}
\label{eq:Mpolydef}
	M_\text{poly}(a,b,c;\mu) =& \abs{1+i\Im(\mu_1-\mu_2)}^a \abs{1+i\Im(\mu_1-\mu_3)}^b \\
	& \qquad \abs{1+i\Im(\mu_2-\mu_3)}^c, \nonumber
\end{align}
this reflects the facts that we will always consider $\Re(\mu)$ fixed and the bounds coming from the Whittaker functions and their residues will depend only on the differences of the $\mu$ variables.
In deriving bounds for the Whittaker functions, we will tend to assume an ordering on the differences, so we must symmetrize the bound to remove this assumption:
\begin{align}
\label{eq:Msymdef}
	M_\text{sym}(\kappa, \delta; \mu) =& \sum_{w\in W} M_\text{poly}\paren{\frac{\kappa-\delta}{2}+\epsilon,\frac{\kappa-\delta}{2}+\epsilon,\delta+\epsilon;\mu^w},
\end{align}
here $\kappa$ is the total exponent, i.e. the sum of the three exponents, and $\delta$ is the minimum exponent; the worst bound always occurs when $\delta$ is the exponent of the smallest difference, and the two larger differences are always of comparable size.
Armed with this notation, we describe the $J_{w,\mu}$ kernel functions as follows:
\begin{prop}
\label{prop:JwHoloAndBd}
	We have each $J_{w,\mu}$ holomorphic on $\Re(\mu) \in \mathcal{A}_w$,
	\begin{align*}
		\mathcal{A}_{w_4} =& \set{\eta:-\frac{1+\Delta}{2} < \eta_1 \le \eta_3 \le \eta_2 < -2\eta_1}, \\
		\mathcal{A}_{w_5} =& \set{\eta:-2\eta_2 < \eta_1 \le \eta_3 \le \eta_2 < 1+\Delta}, \\
		\mathcal{A}_{w_l} =& \biggl\{\eta:-\frac{1+\Delta}{2} < \eta_1 \le \eta_3 \le \eta_2 < 1+\Delta, \\
		& \qquad \eta_1-\eta_2 > -2, -2\eta_2<\eta_1<-\frac{\eta_2}{2}\biggr\},
	\end{align*}
	with the bounds:
	\begin{align*}
		J_{w_4,\mu}(y) \ll& \abs{y}^{-\Re(\mu_1)} M_\text{sym}(5+3\Re(\mu_1), 1+\delta_{w_4}/2; \mu) \\
		J_{w_5,\mu}(y) \ll& \abs{y}^{\Re(\mu_2)} M_\text{sym}(5-3\Re(\mu_2), 1+\delta_{w_5}/2; \mu) \\
		J_{w_l,\mu}(y) \ll& \abs{y_1}^{-\Re(\mu_1)} \abs{y_2}^{\Re(\mu_2)} M_\text{sym}(5 + 2\Re(\mu_1-\mu_2), 1+\delta_{w_l}/2; \mu),
	\end{align*}
	where
	\begin{align*}
		\delta_u =& \Min{\Re(\mu_1)-\Re(\mu_2)-1,3\Re(\mu_1), -3\Re(\mu_2)}, \\
		\delta_{w_4} =& \delta_u+1-\Delta+2\Re(\mu_1), \\
		\delta_{w_5} =& \delta_u+1-\Delta-\Re(\mu_2), \\
		\delta_{w_l} =& \delta_u+1-\Delta+2\Re(\mu_1).
	\end{align*}
\end{prop}
From the above proposition, we see that the total power in the hypothesis \eqref{eq:KuzSimplkBd} can be improved to convergence of the integral
\begin{align}
\label{eq:KuzSimplkBetterBd}
	& \int_{\Re(\mu)=\eta} \abs{\hat{k}(\mu)} M_\text{sym}\paren{\kappa, \delta; \mu} \abs{d\mu}, \\
	& \kappa = \Max{3,\frac{7-3\Delta}{2}}, \nonumber \\
	& \delta = \Max{-\frac{1+3\Delta}{6},\frac{1-11\Delta}{8}}. \nonumber
\end{align}

We next compute a Mellin-Barnes integral representation for $J_{w_l,\mu}$ in section \ref{sec:JwlMellinBarnes}.
Though we strongly suspect that we have not achieved the optimal such representation, it allows us to compute a type of first-term asymptotic for $J_{w_l,\mu}$ in section \ref{sec:JwlAsymptotic}:
\begin{prop}
\label{prop:JwlAsymptotic}
	$J_{w_l,\mu}(y) = \abs{\pi y_1}^{-\mu_1} \abs{\pi y_2}^{\mu_2} K_{w_l}(\mu) + \sum_{j=1}^7 E_{w_l,j}(\mu,y)$, where
	\begin{align*}
		K_{w_l}(\mu) =& \frac{3 \pi^{2(\mu_2-\mu_1)}}{8 \pi^{\frac{17}{2}}} \prod_{(j,k)\in S} \frac{\paren{\frac{\Gamma\paren{\frac{1+\Delta+\mu_k-\mu_j}{2}} \Gamma\paren{\frac{1+\Delta+\mu_j-\mu_k}{2}}}{\paren{(3+2\Delta)^2-\paren{\mu_j-\mu_k}^2}^{\Delta/2}}}}{\Gamma\paren{\frac{1+\mu_k-\mu_j}{2}} \Gamma\paren{\frac{\mu_k-\mu_j}{2}}}, \\
		S =&\set{(1,2),(1,3),(3,2)}
	\end{align*}
	and the $E_{w_l,j}$ are given explicitly by equations \eqref{eq:Ewl1}-\eqref{eq:Ewl7} and satisfy
	\begin{equation}
		E_{w_l,j}(\mu,y) = o\paren{\abs{y_1}^{-\Re(\mu_1)} \abs{y_2}^{\Re(\mu_2)}}
	\label{eq:EwljAsymptotics}
	\end{equation}
	as $y \to 0$ with $\Re(\mu_1)\le\Re(\mu_3)\le\Re(\mu_2)$.
\end{prop}
The asymptotics \eqref{eq:EwljAsymptotics} are actually power-saving bounds over \\ $\abs{y_1}^{-\Re(\mu_1)} \abs{y_2}^{\Re(\mu_2)}$ which are vital to our purposes, but their dependence on $\Re(\mu_1)$ and $\Re(\mu_2)$ is unfortunately quite complicated.
Note that the only zeros of $K_{w_l}(\mu)$ in $\abs{\Re(\mu_i)} < \frac{1+\Delta}{2}$ occur when one of $\mu_2-\mu_3$, $\mu_3-\mu_1$, $\mu_2-\mu_1$ are $-n$ for some $n \ge 0$, and it has no poles on this region.

With this asymptotic of $J_{w_l,\mu}$ and the accompanying explicit error terms, in section \ref{sec:PartialInversion} we produce a type of partial inversion to the $H_{w_l}$ transform:
\begin{thm}
\label{thm:PartialInversion}
	Let $f:(\R^+)^2\to\C$ such that
	\[ \hat{f}(q) := \int_{(\R^+)^2} f(y) y_1^{q_1} y_2^{q_2} \frac{dy_1\, dy_2}{y_1 y_2} \]
	is holomorphic on $-\frac{1}{2}-\epsilon < \Re(q_i) \le 0$ and satisfies \\ $\hat{f}(q) \ll \abs{q_1 q_2 (q_2-q_1)}^{-4}$ there.
	Then we again have the formula \eqref{eq:LisKuznetsovFormula}, where now
	\begin{align*}
		H_{w_l}(k,\psi_m,\psi_n,c) =& \frac{1}{c_1 c_2}f\paren{\frac{\pi c_2 m_1 n_2}{c_1^2}, \frac{\pi c_1 m_2 n_1}{c_2^2}} \\
		& + \frac{1}{c_1 c_2} \sum_{j=1}^{10} F_j\paren{\hat{f};\paren{\frac{\pi c_2 m_1 n_2}{c_1^2}, \frac{\pi c_1 m_2 n_1}{c_2^2}}},
	\end{align*}
	using
	\begin{align}
	\label{eq:PartialInversionkhat}
		\hat{k}(\mu) = \hat{k}_{\mathfrak{q}}(\mu) :=& \frac{1}{2\pi i} \int_{\Re(q) = \mathfrak{q}} \frac{\hat{f}(q)}{K_{w_l}(q_1,-q_2)} k_\text{conv}(\mu,q) \\
		& \frac{\paren{q_1+q_2}^2\paren{2q_1-q_2}\paren{2q_2-q_1}}{\prod_{j=1}^3 \paren{q_1-\mu_j}\paren{q_2+\mu_j}} dq, \nonumber
	\end{align}
	where $k_\text{conv}(\mu,q)$ is chosen in \eqref{eq:kconv} to be holomorphic on $\Re(\mu_i),\Re(q_i) \in (-2,2)$ with $k_\text{conv}(\mu,(\mu_1,-\mu_2)) = 1$, 	and the $F_j$ are given explicitly by equations \eqref{eq:F1}-\eqref{eq:Fjplus3} and satisfy
	\begin{equation}
		F_j\paren{\hat{f};y} = o\paren{\abs{y_1}^{-\mathfrak{q}_1} \abs{y_2}^{-\mathfrak{q}_2}}
	\label{eq:FjAsymptotics}
	\end{equation}
	as $y \to 0$ with $\Re(\mu_1)\le\Re(\mu_3)\le\Re(\mu_2)$.
\end{thm}
Again, \eqref{eq:FjAsymptotics} does not quite do justice to the actual bounds we obtain.
The construction of $\hat{k}_{\mathfrak{q}}(\mu)$ is such that the pole at $q_1=\mu_1$ and $q_2=-\mu_2$ gives the residue
\[ \hat{k}_{\mathfrak{q}}(\mu) = \frac{\hat{f}(\mu_1,-\mu_2)}{K_{w_l}(\mu)} + \text{error terms}, \]
and paired with our first-term asymptotic for $J_{w_l,\mu}$, $H_{w_l}$ looks like the inverse Mellin transform of $\hat{f}$.
The motivation for the existence of $k_\text{conv}$ is to control the convergence of the $\mu$ integral and for ease of proof in the section on bounds; again, the exponent $-4$ is certainly not best possible -- optimal is likely $-\frac{2}{3}-\epsilon$, but it is convenient.

One should note that this is an incomplete generalization to $SL(3,\R)$ of the second form of Kuznetsov's formula on $SL(2,\R)$; it allows us to study sums of Kloosterman sums by applying knowledge of the Fourier-Whittaker coefficients of automorphic forms.
As the asymptotic in \propref{prop:JwlAsymptotic} is for $y_1,y_2 \to 0$, this partial inversion formula is effective when studying sums of Kloosterman sums with $\frac{1}{2} < \frac{\log c_1}{\log c_2} < 2$, i.e. when each of the moduli is at least the square-root of the other.
One would expect that in practice, the remaining sums, i.e. those over $c_1 < \sqrt{c_2}$ or $c_2 < \sqrt{c_1}$, will be small.
Similarly, we expect that the sums of Kloosterman sums for the intermediate Weyl elements $w_4$ and $w_5$ will tend to be small compared to the long-element sum and the trivial term $H_I$.

Lastly, by comparison with the method on $SL(2)$, one might wonder if the multiple error terms at $\Re(\mu) = \paren{-\frac{1}{2},0}$ and $\Re(\mu) = \paren{0,\frac{1}{2}}$ indicate the need for some form of discrete series in the full inversion formula.
This is a somewhat tenuous connection, however.

The formulae used above and in \thmref{thm:SumsOfKloostermanSums} depend strictly on the location of the contours in the Mellin-Barnes integrals for each $F_j$, so we include in section \ref{sec:Bounds} some bounds to demonstrate their absolute convergence at the relevant locations, which are unfortunately not entirely trivial.
The bounds we obtain in \propref{prop:Bounds} are most likely not optimal in their dependence on $\mu$, but they are sufficient for our purposes here.
The Mellin-Barnes representation of $J_{w_l,\mu}$, given by \eqref{eq:JwDef} and \eqref{eq:TwlFinal}, can also be used to improve our bound to
\begin{align}
\label{eq:BetterJwlBound}
	\abs{J_{w_l,\mu}(y)} \ll& \abs{y_1}^{\frac{1}{2}+\epsilon} \abs{y_2}^{\frac{1}{2}+\epsilon} M_\text{sym}\paren{\frac{5}{2},0;\mu},
\end{align}
at $\Re(\mu) = \paren{-\frac{1}{2}-\epsilon,\frac{1}{2}+\epsilon}$ -- essentially we are giving an analytic continuation in the $r$ variables of \eqref{eq:TwlFinal}, but a uniform bound over the range of holomorphy is difficult due to complexity.
Similar, and more striking, bounds for the $w_4$ and $w_5$ kernel functions can be obtained in this manner:
\[ \abs{J_{w_4,\mu}(y)}, \abs{J_{w_5,\mu}(y)} \ll \abs{y}^{\frac{1}{2}+\epsilon} M_\text{sym}\paren{2,0;\mu} \]
at $\Re(\mu) = \paren{-\frac{1}{2}-\epsilon,\frac{1}{2}+\epsilon}$, but we will not require bounds of this strength.

Finally, in section \ref{sec:SumsOfKloostermanSums}, we obtain the results of the introduction.
These results follow directly from the locations of the spectral parameters of the objects in the partial inversion formula \thmref{thm:PartialInversion}, and the absolute convergence of the weight functions in the desired locations using \propref{prop:Bounds}.
We visualize the spectral parameters in figure \ref{fig:SpectralLocation}.

\begin{figure}
	\centering
	\includegraphics{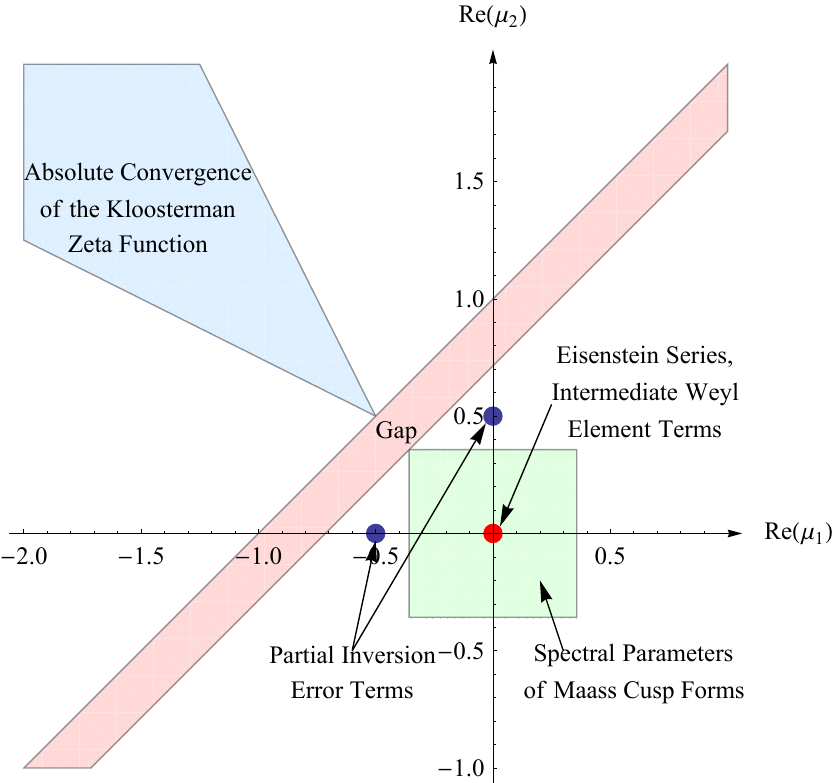}
	\caption{Location of the $SL(3)$ Spectral Parameters}
	\label{fig:SpectralLocation}
\end{figure}

\section{Evaluation of the Integral Transforms}
\label{sec:IntegralTransforms}

\subsection{\texorpdfstring{The $G$ Function}{The G Function}}
We will need a number of elementary results and bounds on the $G$ function to obtain the Kuznetsov formula above, so we collect them here.
We start with $G^*(u,\mu) := \frac{G(u,\mu)}{\Lambda(\mu)}$.
The $G$ function has poles at $u_1 = \mu_i$ and $-u_2=\mu_i$, so up to permutations, we may assume a pole is at $u_1 = \mu_1$, then the residue of $G^*(u,\mu)$ there is given by
\[ G^*_l(u_2,\mu) := 2\pi^{\frac{3}{2}+\mu_1-\mu_3} \frac{\Gamma\paren{\frac{\mu_1-\mu_2}{2}}\Gamma\paren{\frac{\mu_1-\mu_3}{2}}}{\Gamma\paren{\frac{1+\mu_1-\mu_2}{2}}\Gamma\paren{\frac{1+\mu_1-\mu_3}{2}}} \frac{\Gamma\paren{\frac{u_2+\mu_2}{2}}\Gamma\paren{\frac{u_2+\mu_3}{2}}}{\Gamma\paren{\frac{1+\mu_2-\mu_3}{2}}}. \]
For the poles in $u_2$, we let
\[ G^*_r(u_1,\mu) := 2\pi^{\frac{3}{2}+\mu_1-\mu_3} \frac{\Gamma\paren{\frac{\mu_1-\mu_2}{2}}\Gamma\paren{\frac{\mu_3-\mu_2}{2}}}{\Gamma\paren{\frac{1+\mu_1-\mu_2}{2}}\Gamma\paren{\frac{1+\mu_2-\mu_3}{2}}} \frac{\Gamma\paren{\frac{u_1-\mu_1}{2}}\Gamma\paren{\frac{u_1-\mu_3}{2}}}{\Gamma\paren{\frac{1+\mu_1-\mu_3}{2}}}. \]
be the residue at $-u_2=\mu_2$.

The residue at $u_1=\mu_1$ again has poles at $-u_2=\mu_2,\mu_3$, and we assume $-u_2=\mu_2$, giving the residue
\[ G^*_b(\mu) := 4\pi^{\frac{3}{2}+\mu_1-\mu_3} \frac{\Gamma\paren{\frac{\mu_1-\mu_2}{2}}\Gamma\paren{\frac{\mu_1-\mu_3}{2}}\Gamma\paren{\frac{\mu_3-\mu_2}{2}}}{\Gamma\paren{\frac{1+\mu_1-\mu_2}{2}}\Gamma\paren{\frac{1+\mu_1-\mu_3}{2}}\Gamma\paren{\frac{1+\mu_2-\mu_3}{2}}}. \]

We remind the reader of our notation for the polynomial bounds \eqref{eq:Mpolydef} and \eqref{eq:Msymdef}, and add one for the polynomial part of $\Lambda(\mu)$:
\begin{align}
\label{eq:Lambdapolydef}
	\Lambda_\text{poly}(\mu) = M_\text{poly}\paren{\frac{\Re(\mu_1-\mu_2)}{2},\frac{\Re(\mu_1-\mu_3)}{2},\frac{\Re(\mu_2-\mu_3)}{2};\mu}.
\end{align}
Then increasing in complexity, we have bounds for the $G^*$ function and its residues:
\begin{lem}
	\begin{enumerate}
	\item
	\begin{align*}
		G^*_b(\mu) \ll & \frac{M_\text{poly}\paren{\frac{\Re(\mu_1-\mu_2)-1}{2},\frac{\Re(\mu_1-\mu_3)-1}{2},\frac{\Re(\mu_3-\mu_2)-1}{2};\mu}}{\Lambda_\text{poly}(\mu)},
	\end{align*}

	\item Suppose $\mathfrak{u}_2+\Re(\mu_2), \mathfrak{u}_2+\Re(\mu_3) > -1$, then
	\begin{align*}
		& \int_{\Re(u_2)=\mathfrak{u}_2} \abs{G^*_l(u_2,\mu)} \abs{du_2} \ll \\
		& \frac{M_\text{poly}\paren{\frac{\Re(\mu_1-\mu_2)-1}{2},\frac{\Re(\mu_1-\mu_3)-1}{2},\mathfrak{u}_2+\frac{\Re(\mu_2+\mu_3)}{2};\mu}}{\Lambda_\text{poly}(\mu)},
	\end{align*}
	
	\item Suppose $\mathfrak{u}_1-\Re(\mu_1), \mathfrak{u}_1-\Re(\mu_3) > -1$, then
	\begin{align*}
		& \int_{\Re(u_1)=\mathfrak{u}_1} \abs{G^*_r(u_1,\mu)} \abs{du_1} \ll \\
		& \frac{M_\text{poly}\paren{\frac{\Re(\mu_1-\mu_2)-1}{2},\mathfrak{u}_1-\frac{\Re(\mu_1+\mu_3)}{2},\frac{\Re(\mu_3-\mu_2)-1}{2};\mu}}{\Lambda_\text{poly}(\mu)},
	\end{align*}
	
	\item
	Suppose $\mathfrak{u}_1-\Re(\mu_i), \mathfrak{u}_2+\Re(\mu_i) > -1-\epsilon$, then
	\begin{align*}
		& \int_{\Re(u)=\mathfrak{u}} \abs{G^*(u,\mu)} \abs{du} \ll \frac{M_\text{sym}\paren{\mathfrak{u}_1+\mathfrak{u}_2-\frac{1}{2}, \delta; \mu}}{\Lambda_\text{poly}(\mu)},
	\end{align*}
	where
	\[ \delta = \min_i \set{\frac{2\mathfrak{u}_1+\Re(\mu_i)}{2},\frac{\mathfrak{u}_1-\Re(\mu_i)}{2},\frac{2\mathfrak{u}_2-\Re(\mu_i)}{2},\frac{\mathfrak{u}_2+\Re(\mu_i)}{2}}. \]
\label{lem:GStarIntegral}
	\end{enumerate}
\end{lem}
Here we have taken some care to separate the polynomial part of $\Lambda(\mu)$, as it will cancel with that of $\Lambda(-\mu)$.
Notice that for $\mathfrak{u} = \Re(\mu_1,-\mu_2) \pm \epsilon$ with $\Re(\mu_1) \le \Re(\mu_3) \le \Re(\mu_2)$, the bound in part $(d)$ dominates the other three if we replace
\[ \delta = \frac{1}{2}\Min{\Re(\mu_1)-\Re(\mu_2)-1,3\Re(\mu_1), -3\Re(\mu_2)}. \]
We defer the proof of this lemma until section \ref{sec:GFuncBd}.

\subsection{The General Term}
\label{sec:TheGeneralTerm}
In Li's construction of the Kuznetsov formula, the final step involved integrating away some extra variables on the spectral side, using Stade's formula:
\begin{thm}[Stade]
\label{thm:StadesFormula}
	For $\Re(s) \ge 1$,
	\begin{align*}
		& \int_{Y(\R)} W^*(y,\mu,\psi_{11}) W^*(y,\mu',\psi_{11}) y_1^{2s} y_2^{s} \, dy = \\
		& \frac{1}{4 \pi^{3s} \Gamma\paren{\frac{3s}{2}}} \prod_{j=1}^3 \prod_{k=1}^3 \Gamma\paren{\frac{s+\mu_j+\mu_k'}{2}}.
	\end{align*}
\end{thm}
Li applies this at $s=1$, but this will make it impossible to obtain a function which is nicely holomorphic in the region we require.
Instead we will apply Stade's formula at $s=1+\Delta$, which results in a weight function on the spectral side which is actually too large.
So we replace $\hat{k}$ with $\hat{k}(\mu) = \tilde{k}(\mu) k_\text{adj}(\mu)$, with $k_\text{adj}$ as in \eqref{eq:kadjdef}, and the spectral side is now of the correct magnitude as a function of $\tilde{k}$.
We have chosen $k_\text{adj}$ symmetric in $\mu$ and invariant under $\mu \mapsto -\mu$ hence also $\mu \mapsto \wbar{\mu}$ for the Langlands parameters in the spectral decomposition and so $k_\text{adj}$ is real there.
On the spectral side, we replace $C(\mu) \mapsto C^*(\mu)$, which is now the collection of gamma functions coming from applying Stade's formula at $s=1+\Delta$ times $k_\text{adj}$ so that $C(\mu) \asymp C^*(\mu)$, and $C^*(\mu)$ is real in the spectral decomposition.

Before we truly start the simplification process, we must engage in a series of transformations:
Since $\psi_m(x) = \psi_{11}(m x m^{-1})$, and $\abs{m} m^{-1} \in V$, up to a multiple of $-1$, conjugating by $m^{-1}$ and by $tn^{-1}$ we have
\begin{align*}
	H_w =& \frac{2\abs{m_1 m_2 n_1 n_2}}{\pi (m_1 m_2)^2 C_w(n)} \int_{Y(\R)} \int_{U(\R)} \int_{\wbar{U}_w(\R)} k\paren{t^{-1} x^{-1} \alpha w tx'} \\
	& \qquad \psi_{11}(x) \wbar{\psi_t(x')} dx'\,dx\,C_w(t) t_1^{2+2\Delta} t_2^{1+\Delta}\, dt,
\end{align*}
where
\[ \alpha = m c w n^{-1} w^{-1} \equiv \Matrix{\alpha_1 \alpha_2 \\ & \alpha_1 \\ &&1} \pmod{\R^+}, \]
and $C_w(y)$ is the Jacobian of the change of variables $u \mapsto yuy^{-1}$ for $u \in \wbar{U}_w(\R)$.
Note that $\alpha_1$ and $\alpha_2$ may be negative!
Now interchange the $x$ and $x'$ integrals.
Then if $w x' \equiv x^* y^* \pmod{K}$ and $t^w = wtw^{-1}$, we may translate and invert $x^{-1} (\alpha t^w) x^* (\alpha t^w)^{-1} \mapsto x$.
\begin{align*}
	H_w =& \frac{2\abs{m_1 m_2 n_1 n_2}}{\pi (m_1 m_2)^2 C_w(n)} \int_{Y(\R)} \int_{\wbar{U}_w(\R)} \int_{U(\R)} k\paren{t^{-1} x \alpha t^w y^*} \wbar{\psi_{11}(x)} \\
	& \qquad \psi_{\alpha t^w}(x^*) \wbar{\psi_t(x')} dx\,dx' C_w(t) t_1^{2+2\Delta} t_2^{1+\Delta}\, dt.
\end{align*}
Since $k$ is a function on $K\backslash G/K$ (essentially the space of diagonal matrices by the singular value decomposition), it is invariant under transposition of its argument, so send $(\alpha t^w y^*)^{-1} x (\alpha t^w y^*) \mapsto x$ and transpose, giving
\begin{align*}
	H_w =& \frac{2\abs{m_1 m_2 n_1 n_2}}{\pi (m_1 m_2)^2 C_w(n)} \int_{Y(\R)} \int_{\wbar{U}_w(\R)} \int_{U(\R)} k\paren{\trans{x} \alpha t^w y^* t^{-1}} \wbar{\psi_{\alpha t^w y^*}}(x) dx \\
	& \qquad \psi_{\alpha t^w}(x^*) \wbar{\psi_t}(x') p_{2\rho}(\alpha t^w y^*) dx'\,C_w(t) t_1^{2+2\Delta} t_2^{1+\Delta} \, dt.
\end{align*}

We wish to apply spherical inversion, which will require some care with respect to convergence of the integrals, but the integral in $x$ can be evaluated explicitly as in the following lemma.
\begin{lem}[Fourier Transform of the Spherical Function]
\label{lem:FourierTransformoftheSphericalFunction} \ 

	For $\Re(\mu) = \paren{-\frac{1}{3}-\epsilon,-\frac{1}{3}-2\epsilon}$, let
	\[ X(y,\mu,\psi) = \int_{U(\R)} h_\mu(\trans{x}y) \psi(x) dx, \]
	where the integrals over $x_1$ and $x_2$ are taken in the limit sense \\ $\int_{-\infty}^\infty=\lim_{R\to\infty} \int_{-R}^R$, then
	\[ X(y,\mu,\psi) = \kappa W(y^{-1},-\mu,\psi) W(I,\mu,\psi), \]
	and the constant is given by
	\[ \frac{1}{\kappa} = \prod_{1\le i<j \le 3} B\paren{\frac{1}{2}, \frac{j-i}{2}} = 2\pi^2. \]
\end{lem}
Note: We expect this formula to hold on $SL(n,\R)$ for arbitrary $n$, but the interchange of integrals might be difficult to justify.

Applying spherical inversion, we may shift the integrals in $\mu$ to \\ $\Re(\mu) = \paren{-\frac{1}{3}-\epsilon,-\frac{1}{3}-2\epsilon}$ and apply the above lemma.
Note that, despite appearances, we have an expression for $\abs{c_3(\mu)}^2$ which is analytic in $\mu$.
After moving $\mu$ back to the 0 lines, we have
\begin{align*}
	H_w &= -\frac{\abs{m_1 m_2 n_1 n_2}}{64\pi^7 (m_1 m_2)^2 C_w(n)} \int_{Y(\R)} \int_{\wbar{U}_w(\R)} \int_{\Re(\mu)=(0,0)} \frac{\hat{k}(\mu)}{\abs{c_3(\mu)}^2} \\
	& \qquad W((\alpha t^w y^*)^{-1} t,-\mu,\wbar{\psi_{\alpha t^w y^*}}) W(I,\mu,\wbar{\psi_{\alpha t^w y^*}}) d\mu\, \\
	& \qquad \psi_{\alpha t^w}(x^*) \wbar{\psi_t}(x') p_{2\rho}(\alpha t^w y^*) dx'\,C_w(t) t_1^{2+2\Delta} t_2^{1+\Delta} \, dt.
	% 2/pi from previous, 1/2 pi^2 from \kappa, -1/48 pi^4 from spherical inversion %
\end{align*}
We then return $\alpha t^w y^*$ to the argument of the Whittaker function:
\begin{align}
\label{eq:BeforeMellinExpansionOfWhittaker}
	H_w &= -\frac{\abs{m_1 m_2 n_1 n_2}}{64\pi^7 (m_1 m_2)^2 C_w(n)} \int_{Y(\R)} \int_{\wbar{U}_w(\R)} \int_{\Re(\mu)=(0,0)} \frac{\hat{k}(\mu)}{\abs{c_3(\mu)}^2} \\
	& \qquad W(t,-\mu,\psi_{11}) W(\alpha t^w y^*,\mu,\psi_{11}) d\mu\, \nonumber \\
	& \qquad \psi_{\alpha t^w}(x^*) \wbar{\psi_t}(x') dx'\,C_w(t) t_1^{2+2\Delta} t_2^{1+\Delta} \, dt. \nonumber
\end{align}

For all but the trivial term, we will use the Mellin expansion of the second Whittaker function.
We will shift the lines of integration in this Mellin expansion $\Re(u) \mapsto -\frac{1}{2}-10\epsilon$, picking up poles at $u_1 = \mu_i$ and $-u_2=\mu_i$.
As $\hat{k}$, $c_3$, and the product
\[ W(z,-\mu,\psi_{11})W(z',\mu,\psi_{11}) = \frac{W^*(z,-\mu,\psi_{11})W^*(z',\mu,\psi_{11})}{\Lambda(-\mu)\Lambda(\mu)} \]
are invariant under permutations of $\mu$, we may collect like terms, leaving us with four pieces:
The pole at $u_1 = \mu_2$, $u_2=-\mu_2$; the pole at $u_1=\mu_1$ with an integral along $\Re(u_2) = -\frac{1}{2}-10\epsilon$; the pole at $u_2=-\mu_2$ with an integral along $\Re(u_1) = -\frac{1}{2}-10\epsilon$; the double integral along $\Re(u) = -\frac{1}{2}-10\epsilon$.
Then we shift $\Re(\mu_1,\mu_2) \mapsto \paren{-\frac{1}{2}-9\epsilon,\frac{1}{2}+9\epsilon}$, which is possible since $\frac{\abs{c_n(\mu)}^{-2}}{\Lambda(\mu)\Lambda(-\mu)}$ has no poles.
Thus we are evaluating
\begin{align*}
	H_w = \frac{1}{(2 \pi i)^2 c_1 c_2} \int_{\Re(\mu)=\paren{-\frac{1}{2}-9\epsilon,\frac{1}{2}+9\epsilon}} \tilde{k}(\mu) J_{w,\mu}(\alpha) \, d\mu,
\end{align*}
with $J_{w,\mu}$ given in \eqref{eq:JwDef}.
Note that $T_w(u,y)$, initially given by
\begin{align*}
	T_w(u,y) =& \frac{\abs{m_1 m_2 n_1 n_2} c_1 c_2}{\pi^2 (m_1 m_2)^2 C_w(n)} \int_{Y(\R)} W(t,-\mu,\psi_{11}) X'_w(u,y, t) \\
	& \qquad (\pi\abs{y_1} t^w_1)^{1-u_1} (\pi\abs{y_2} t^w_2)^{1-u_2} C_w(t) t_1^{2+2\Delta} t_2^{1+\Delta} \, dt,
\end{align*}
simplifies to \eqref{eq:TwFinal}, since for every $w$, we have
\[ \frac{\abs{m_1 m_2 n_1 n_2} c_1 c_2}{\pi^2 (m_1 m_2)^2 C_w(n)} (\pi \beta_1) (\pi \beta_2) = 1, \quad C_w(t) (t^w_1 t^w_2) t_1^{2+2\Delta} t_2^{1+\Delta} = t_2^{3+2\Delta} t_1^{2+\Delta}. \]

We justify the interchange of $\mu$ and $x'$ integrals by explicitly computing $y^*$, which shows that, in general, $X'_w$ converges absolutely for some region in $\Re(u_1),\Re(u_2) < 0$.
This will further show that the $t$ integral converges absolutely as well as the sum of Kloosterman sums.
This step was a technical necessity, as we did not know it was safe to pull the $t$ integral inside the sum of Kloosterman sums until right now, but having done so, we may forget about the sum of Kloosterman sums.

Note that the function $X'_w$ is a type of generalized hypergeometric function; for fixed $\sgn(y)$, it is a function of four variables $\abs{y_1}, \abs{y_2}, t_1, t_2$ with two parameters $u_1$ and $u_2$.
We will leave this function in its current form, and bound it as in \cite{GS01}.

If $X'_w$ is absolutely bounded on some open set $u \in \mathcal{B}_w$, then by \propref{prop:MellinAbsWhittaker} below, we have absolute (and uniform on compact subsets) convergence of $T_w$ on $\Re(\mu) \in \mathcal{B}_w$,
\begin{align*}
	\mathfrak{u} \in& \set{(\mathfrak{u}_1,\mathfrak{u}_2):2-\mathfrak{u}^w_1+2\Delta > \max_i \Re(\mu_i), 1-\mathfrak{u}^w_2+\Delta > -\min_i \Re(\mu_i)}.
\end{align*}
Thus it is holomorphic and $T_w \ll_\mu \beta_1^{-\Re(u_1)} \beta_2^{-\Re(u_2)}$ where the constant depends polynomially on $\mu$.
Finally, we have the same for the $u$ integrals by \lemref{lem:GStarIntegral}.
Specifically, we choose $\mathfrak{u} = \Re(\mu_1,-\mu_2)-\epsilon$ and apply the comment following \lemref{lem:GStarIntegral} to the $u$ integrals of the $G^*$ function and its residues; the ratio $\frac{k_\text{adj}(\mu)}{\abs{c_3(\mu)}^2}$ is within a constant of $\prod_{j<k} \abs{\mu_j-\mu_k}^{1-\Delta}$.
Thus $J_{w,\mu}$ is holomorphic on $\Re(\mu) \in \mathcal{A}_w$ with
\begin{align*}
	\mathcal{A}_w =& \mathcal{B}_w \cap \big\{\eta: -1+\mathfrak{u}^w_2(\eta)-\Delta <\eta_1 \le \eta_3 \le \eta_2 < 2-\mathfrak{u}^w_1(\eta)+2\Delta, \\
	& \eta_1-\eta_2 > -2\big\},
\end{align*}
and the combined bound becomes
\begin{align*}
	J_{w,\mu}(\alpha) \ll& \beta_1^{-\Re(u_1)} \beta_2^{-\Re(u_2)} M_\text{sym}\paren{\kappa+3,\delta+1;\mu},
\end{align*}
where
\begin{align*}
	\kappa =& \Re(\mu_1-\mu_2)-\mathfrak{u}_1^w(\mu)-\mathfrak{u}_2^w(\mu)+2, \\
	\mathfrak{u}(\mu) =& \Re(\mu_1,-\mu_2), \\
	\delta =& \frac{1}{2}\Min{\Re(\mu_1)-\Re(\mu_2)-1, 3\Re(\mu_1),-3\Re(\mu_2)} \\
	& + \frac{1}{2}\min\{4-2\mathfrak{u}_1^w+2\Delta+\Re(\mu_1),2-\mathfrak{u}_1^w-\Re(\mu_2),\\
	& \qquad 2-2\mathfrak{u}_2^w-\Re(\mu_2),1-\mathfrak{u}_2^w-\Delta+\Re(\mu_1)\}.
\end{align*}

We conclude the construction of the formula with the following proposition:
\begin{prop}
\label{prop:MellinAbsWhittaker}
	Suppose $\mathfrak{s}_1-\Re(\mu_i), \mathfrak{s}_2+\Re(\mu_i) > 0$, then
	\begin{align*}
		& \int_{Y(\R)} \abs{W(t,-\mu,\psi_{11})} t_1^{\mathfrak{s}_1+1} t_2^{\mathfrak{s}_2+1} \, dt \ll \frac{M_\text{sym}\paren{\mathfrak{s}_1+\mathfrak{s}_2-\frac{1}{2}, \delta; \mu}}{\Lambda_\text{poly}(-\mu)},
	\end{align*}
	where
	\[ \delta = \min_i \set{\frac{2\mathfrak{s}_1+\Re(\mu_i)}{2},\frac{\mathfrak{s}_1-\Re(\mu_i)}{2},\frac{2\mathfrak{s}_2-\Re(\mu_i)}{2},\frac{\mathfrak{s}_2+\Re(\mu_i)}{2}}. \]
\end{prop}
\begin{proof}
Split the integrals at $t_1=1$,$t_2=1$, and apply \lemref{lem:GStarIntegral} part (4) to the Mellin expansion of the Whittaker function \eqref{eq:WhittakerMellinExpand} with $\mathfrak{u} = \mathfrak{s} \pm \epsilon$, and the sign of $\pm \epsilon$ determined for each coordinate as necessary for convergence:
This choice of parameters makes the total exponent of $t_1$ and $t_2$ equal to $-1\pm\epsilon$
\end{proof}

\subsection{Fourier Transform of the Spherical Function}
The proof of \lemref{lem:FourierTransformoftheSphericalFunction} has three parts:
First, we give an integral formula for the spherical function; this is a slight extension of a formula in Terras \cite[3.30]{T02}.
Second, we give a formal argument that interchanging integrals gives us a product of two integrals, and some translation shows these integrals are, in fact, each a Jacquet-Whittaker function.
Lastly, we justify the interchange of integrals in the absence of absolute convergence; essentially, we find an integral formula for the Jacquet-Whittaker function having a slightly larger region of absolute convergence.

We define the $K$-part function on $G$ as $K(xyk) = k$, then the power function identity \cite[3.17]{T02}
\[ p_{\rho+\mu}(K(\trans{x})z) = p_{\rho+\mu}(\trans{x}z) p_{-\rho-\mu}(\trans{x}) \]
comes from the decomposition $\trans{x} = x_1 y_1 k_1$:
\[ p_{\rho+\mu}(\trans{x}z) = p_{\rho+\mu}(x_1 y_1 k_1 z) = p_{\rho+\mu}(y_1) p_{\rho+\mu}(k_1 z). \]
We also need the change of variables formula \cite[Lemma 4.3.2]{T02}
\[ \int_{K/V} f(\bar{k}) d\bar{k} = \kappa \int_{U(\R)} f(K(\trans{x})V) p_{2\rho}(\trans{x}) dx, \]
where the measure on $K/V$ is again normalized to $\int_{K/V} d\bar{k} = 1$
Then we expand
\begin{align*}
	h_\mu(z) &= \int_{K/V} \sum_{v\in V} p_{\rho+\mu}(\bar{k}vz) \frac{d\bar{k}}{\abs{V}} \\
	&= \int_{K/V} p_{\rho+\mu}(\bar{k}z) d\bar{k} \\
	&= \kappa \int_{U(\R)} p_{\rho+\mu}(K(\trans{u})z) p_{2\rho}(\trans{u}) du \\
	&= \kappa \int_{U(\R)} p_{\rho+\mu}(\trans{u}z) p_{\rho-\mu}(\trans{u}) du,
\end{align*}
by substituting $v\bar{k}v \mapsto \bar{k}$ in the first integral.

Working formally, we apply this integral representation of $h_\mu$ to $X$, so we have
\[ X(y,\mu,\psi) = \kappa \int_{U(\R)} \int_{U(\R)} p_{\rho+\mu}(\trans{u}\trans{x}y) p_{\rho-\mu}(\trans{u}) \, du \, \psi(x) dx. \]
We then interchange the integrals and send $xu\mapsto x$, giving
\begin{align*}
	X = \kappa \int_{U(\R)} p_{\rho-\mu}(\trans{u}) \wbar{\psi(u)} \, du \int_{U(\R)} p_{\rho+\mu}(\trans{x}y) \, \psi(x) dx.
\end{align*}
Note that our assumptions on $\mu$ mean that the $u$ integral converges absolutely, while the $x$ integral does not; as no choice of $\mu$ can make both converge absolutely, we will carefully justify this interchange of integrals below.

To finish, we note a symmetry of the power function \cite[Prop. 4.2.1 (4)]{T02}:
With $\mu^{w_l} = (\mu_3,\mu_2,\mu_1)$, we have
\[ p_\mu(xy) = p_\mu(y) = p_{-\mu^{w_l}}(w_l y^{-1}) = p_{-\mu^{w_l}}(w_l \trans{\paren{xy}}^{-1}), \]
so the Jacquet-Whittaker function may be written as
\[ W(y^{-1},-\mu^{w_l},\psi) = \int_{U(\R)} p_{\rho+\mu}(\trans{x} y) \psi(x) dx, \]
by sending $x \mapsto vx^{-1}v$ for $v=\SmallMatrix{1\\&-1\\&&1}$ and noticing that $-\rho^{w_l} = \rho$.
Lastly, the product
\[ W(z,-\mu,\psi_{11})W(z',\mu,\psi_{11}) = \frac{W^*(z,-\mu,\psi_{11})W^*(z',\mu,\psi_{11})}{\Lambda(-\mu)\Lambda(\mu)} \]
is permutation-invariant in $\mu$, so we may replace $\mu^{w_l} \mapsto \mu$.
This gives the lemma.

Now to complete the proof, we must justify the interchange of the $x$ and $u$ integrals in the absence of absolute convergence.
Interpretting the $x_1$ and $x_2$ integrals in the limit sense, we have
\begin{align*}
	X =& \kappa \lim_{R\to\infty} \int_{[-R,R]^2\times\R} \int_{U(\R)} p_{\rho+\mu}(\trans{u}\trans{x}y) p_{\rho-\mu}(\trans{u}) \, du \, \psi(x) dx \\
	=& \kappa p_{-\rho+\mu}(y) \lim_{R\to\infty} \int_{U(\R)} \int_{\mathcal{X}(u,y,R)\times\R} p_{\rho+\mu}(\trans{x}) \psi^*(x) \, dx \, p_{\rho-\mu}(\trans{u}) \wbar{\psi(u)} \, du,
\end{align*}
where $\mathcal{X}(u,y,R)$ is the result of applying the inverse of $xu\mapsto x$, and $\trans{x} \mapsto y\trans{x}y^{-1}$ to the box $[-R,R]^2$, and $\psi^*(x) = \psi(y^{-1} x y)$.
Note the integrals inside the limit converge absolutely; thus we need only to rearrange the $x$ integral into an absolutely convergent form allowing us to pull the limit inside, by dominated convergence.
Then the above formal argument shows that the $x$ integral will actually converge to the appropriate Whittaker function by analytic continuation, and the $u$ integral will converge to the appropriate Whittaker function by definition.

To that end, we separate the $x_3$ integral
\begin{align*}
	&X_3(x_1, x_2, \mu) = \int_{\R} p_{\rho+\mu}(\trans{x}) dx_3 \\
	&= \int_{\R} \paren{1+x_1^2+x_3^2}^{-\frac{1+\mu_1+2\mu_2}{2}} \paren{1+x_2^2+(x_3-x_1 x_2)^2}^{-\frac{1+\mu_1-\mu_2}{2}} dx_3,
\end{align*}
and for convenience, we write $X_3(s_1, s_2) = X_3(x_1, x_2, \mu)$ where $s_1 = -\frac{1+\mu_1+2\mu_2}{2}$ and $s_2 = -\frac{1+\mu_1-\mu_2}{2}$.

A quick and useful bound for $X_3$ comes from applying H\"older's inequality:
\begin{align*}
	& \abs{X_3(s_1,s_2)} \ll_s \paren{1+x_1^2}^{\Re(s_1)+\frac{1}{2p}} \paren{1+x_2^2}^{\Re(s_2)+\frac{p-1}{2p}}, \nonumber
\end{align*}
assuming $\Re(s_1) < -\frac{1}{2p}$,$\Re(s_2)<-\frac{p-1}{2p}$.
For the moment, we may take $p=2$.

Substituting $x_3 \mapsto x_1 x_3$ gives
\[ X_3 = x_1^{2(s_1+s_2)+1} \int_{\R} \paren{x_1^{-2}+1+x_3^2}^{s_1} \paren{\frac{1+x_2^2}{x_1^2}+(x_3-x_2^2)^2}^{s_2} dx_3, \]
so
\begin{align*}
	\pderv{X_3}{x_1} =& \frac{2(s_1+s_2)+1}{x_1} X_3(s_1,s_2)-2\frac{s_1}{x_1} X_3(s_1-1,s_2) \\
	& \qquad -2s_2\frac{1+x_2^2}{x_1}X_3(s_1,s_2-1),
\end{align*}
and similarly for $\pderv{X_3}{x_2}$.

By comparing the powers of $x_1$ and $x_2$ in each of the partial derivatives of $X_3$ against the given bound for the corresponding $X_3(s_1-a,s_2-b)$, we see that integration by parts causes problems near zero, but will give us convergence on an integral which is bounded away from zero, so we now split the plane into four regions (nine total components) as $x_1$ and $x_2$ have magnitude smaller or larger than 1.
On the regions unbounded in $x_1$, we integrate by parts twice in $x_1$, similarly for $x_2$, and the boundary lines.

Note that the only dependence on $u$ in the $x$ integral is to position the center of the box $\mathcal{X}$.
The integrals over the various regions, boundary lines and points now converge absolutely, and by dominated convergence, we may move the limit inside the $s$ and $u$ integrals to obtain an absolutely convergent integral.
This causes the $x$ and $u$ integrals separate, and we also undo the substitution $\trans{x} \mapsto y\trans{x}y^{-1}$.
Thus the rearranged $x$ integral (now consisting of integrals over nine regions, twelve lines, and four points in the $x_1, x_2$ plane) is, by construction, an analytic continuation of the Whittaker function, originally defined on $\Re(s_1)<-\frac{1}{2},\Re(s_2)<-\frac{1}{2},\Re(s_1+s_2)<-\frac{3}{2}$ to the region $\Re(s_1)<\frac{1}{2},\Re(s_2)<\frac{1}{2},\Re(s_1+s_2)<\frac{1}{2}$, which in particular includes $\Re(\mu) = \paren{-\frac{1}{3}-\epsilon,-\frac{1}{3}-2\epsilon}$.

\subsection{The Weyl element terms of \thmref{thm:KuznetsovSimplification} and \propref{prop:JwHoloAndBd}.}
\label{sec:KSimpWeylTerms}
We now apply the results of section \ref{sec:TheGeneralTerm} to each of Weyl element terms one at a time to obtain the explicit forms of $J_I$ and $\alpha$ given in \thmref{thm:KuznetsovSimplification}, and to finish \propref{prop:JwHoloAndBd}, we must compute the set $\mathcal{B}_w$ described above.
As mentioned above, our construction of the $J_{w,\mu}$ functions relies on the absolute convergence the $X'_w$ integral, as well.

\subsubsection{Trivial Element Term}
Only occurs when $\abs{m} = \abs{n}$ and only for the $c=I$ term; the integral over $\wbar{U}_w(\R)$ is trivial as well.
$C_w(y)$ is just $1$ since we didn't actually do any substituting, $\alpha=I$, and $x^*=I$, $y^* = I$ since $I I$ is already of the form $x^* y^*$, so pulling the $t$ integral inside in \eqref{eq:BeforeMellinExpansionOfWhittaker} (justified by the absolute convergence of the interchanged form) gives
\begin{align*}
	H_w &= -\frac{\abs{m_1 m_2 n_1 n_2}}{64\pi^7 (m_1 m_2)^2} \int_{\Re(\mu)=(0,0)} \frac{\hat{k}(\mu)}{\abs{c_3(\mu)}^2 \Lambda(\mu) \Lambda(-\mu)} \\
	& \qquad \int_{Y(\R)} W^*(t,-\mu,\psi_{11}) W^*(t,\mu,\psi_{11}) t_1^{2+2\Delta} t_2^{1+\Delta} \, dt \, d\mu \\
	&= -\frac{1}{128 \pi^6} \int_{\Re(\mu)=(0,0)} \frac{\hat{k}(\mu)}{\abs{c_3(\mu)}^2 \Lambda(\mu) \Lambda(-\mu) C^*(\mu)} d\mu \\
	&= -\frac{1}{128 \pi^6} \int_{\Re(\mu)=(0,0)} \tilde{k}(\mu) J_I(\mu) \, d\mu.
\end{align*}

\subsubsection{Long Element Term}
The computational data that is required is
\[ \wbar{U}_w(\R) = U(\R), \quad C_w(y) = (y_1 y_2)^2, \quad t^w = \paren{\frac{1}{t_2},\frac{1}{t_1}}, \quad u^w=\paren{-u_2,-u_1} \]
\[ \alpha_1 = \frac{c_2 m_1 n_2}{c_1^2}, \qquad \alpha_2 = \frac{c_1 m_2 n_1}{c_2^2}, \]
\[ x^*_1 = -\frac{x_2'+x_1'x_3'}{1+x_2'^2+x_3'^2}, \qquad x^*_2 = -\frac{x_1'+x_2'(x_1'x_2'-x_3')}{1+x_1'^2+(x_1'x_2'-x_3')^2}, \]
\[ y^*_1 = \frac{\sqrt{1+x_1'^2+(x_1'x_2'-x_3')^2}}{1+x_2'^2+x_3'^2}, \qquad y^*_2 = \frac{\sqrt{1+x_2'^2+x_3'^2}}{1+x_1'^2+(x_1'x_2'-x_3')^2}, \]
so that we have
\begin{align}
\label{eq:XwlFinal}
	X'_{w_l}(u,v,\beta,t) &= \int_{U(\R)} \e{-v_1\frac{\beta_1}{t_2} x_1^*-v_2\frac{\beta_2}{t_1} x_2^*+t_1 x_1'+t_2 x_2'} \\
	& \qquad (y^*_1)^{1-u_1} (y^*_2)^{1-u_2}\,dx'. \nonumber
\end{align}

As the product $(y^*_1)^{1-u_1} (y^*_2)^{1-u_2}$ is actually
\[ \paren{1+x_2'^2+x_3'^2}^{\frac{-1+2u_1-u_2}{2}}\paren{1+x_1'^2+(x_1'x_2'-x_3')^2}^{\frac{-1-u_1+2u_2}{2}}, \]
we see that $X'_{w_l}$ converges absolutely (and uniformly on compact subsets) on
\[ \Re(u) \in \set{(\mathfrak{u}_1,\mathfrak{u}_2):2\mathfrak{u}_1-\mathfrak{u}_2<0, -\mathfrak{u}_1+2\mathfrak{u}_2<0}. \]
Thus it is holomorphic and bounded there.
The results of \propref{prop:JwHoloAndBd} for the $w_l$ term follow by simplification.

\subsubsection{\texorpdfstring{The $w_4$ Term}{The w4 Term}}
The computational data that is required is
\[ \wbar{U}_w(\R) = \set{\Matrix{1&x_2&x_3\\&1&0\\&&1}:x_2,x_3\in\R}, \]
\[ C_w(y) = y_1 y_2^2, \quad t^w = \paren{\frac{1}{t_1 t_2},t_1}, \quad u^w = \paren{u_2-u_1,-u_1}, \]
\[ \alpha_1 = \frac{c_2 m_1 n_1 n_2}{c_1^2} = \frac{m_1 m_2^2 n_2}{c_2^3 n_1}, \qquad \alpha_2 = \frac{c_1 m_2}{c_2^2 n_1} = 1 \]
\[ x^*_1 = \frac{x_3'}{1+x_2'^2+x_3'^2}, \qquad x^*_2 = -\frac{x_2'x_3'}{1+x_2'^2}, \]
\[ y^*_1 = \frac{\sqrt{1+x_2'^2}}{1+x_2'^2+x_3'^2}, \qquad y^*_2 = \frac{\sqrt{1+x_2'^2+x_3'^2}}{1+x_2'^2}, \]
so that we have
\begin{align}
\label{eq:Xw4Final}
	X'_{w_4}(u,v,\beta,t) &= \int_{\wbar{U}_w(\R)} \e{-v_1\frac{\beta_1}{t_1 t_2} x_1^*-t_1 x_2^*+t_2 x_2'} \\
	& \qquad (y^*_1)^{1-u_1} (y^*_2)^{1-u_2}\,dx'. \nonumber
\end{align}

Again, the product $(y^*_1)^{1-u_1} (y^*_2)^{1-u_2}$ is
\[ \paren{1+x_2'^2}^{\frac{-1-u_1+2u_2}{2}} \paren{1+x_2'^2+x_3'^2}^{\frac{-1+2u_1-u_2}{2}}, \]
so $X'_{w_4}$ converges in absolute value on
\[ \Re(u) \in \set{(\mathfrak{u}_1,\mathfrak{u}_2):2\mathfrak{u}_1-\mathfrak{u}_2<0,\mathfrak{u}_1+\mathfrak{u}_2<0}. \]
Thus it is holomorphic and bounded there.

\subsubsection{\texorpdfstring{The $w_5$ Term}{The w5 Term}}
The computational data that is required is
\[ \wbar{U}_w(\R) = \set{\Matrix{1&0&x_3\\&1&x_1\\&&1}:x_1,x_3\in\R}, \]
\[ C_w(y) = y_1^2 y_2, \quad t^w = \paren{t_2,\frac{1}{t_1 t_2}}, \quad u^w = \paren{-u_2,u_1-u_2}, \]
\[ \alpha_1 = \frac{c_2 m_1}{c_1^2 n_2} = 1, \qquad \alpha_2 = \frac{c_1 m_2 n_1 n_2}{c_2^2} = \frac{m_1^2 m_2 n_1}{c_1^3 n_2} \]
\[ x^*_1 = -\frac{x_1'x_3'}{1+x_1'^2}, \qquad x^*_2 = \frac{x_3'}{1+x_1'^2+x_3'^2}, \]
\[ y^*_1 = \frac{\sqrt{1+x_1'^2+x_3'^2}}{1+x_1'^2}, \qquad y^*_2 = \frac{\sqrt{1+x_1'^2}}{1+x_1'^2+x_3'^2}, \]
so that we have
\begin{align}
\label{eq:Xw5Final}
	X'_{w_5}(u,v,\beta,t) &= \int_{\wbar{U}_w(\R)} \e{-t_2 x_1^*-v_2\frac{\beta_2}{t_1 t_2} x_2^*+t_1 x_1'} \\
	& \qquad (y^*_1)^{1-u_1} (y^*_2)^{1-u_2}\,dx'. \nonumber
\end{align}
Symmetrically with $X'_{w_4}$, $X'_{w_5}$ is holomorphic and bounded on
\[ \Re(u) \in \set{(\mathfrak{u}_1,\mathfrak{u}_2):-\mathfrak{u}_1+2\mathfrak{u}_2<0,\mathfrak{u}_1+\mathfrak{u}_2<0}. \]

\section{\texorpdfstring{The Mellin-Barnes Representation of the $J_{w_l,\mu}$ Function}{The Mellin-Barnes Representation of the Jwl,mu Function}}
\label{sec:JwlMellinBarnes}

For $y \in Y(\R)$, we write $v = \sgn(y)$ and $\beta = \abs{y}$.
Starting with the $X'_{w_l}$ integral, we wish to separate the three $x'$ variables, so we start by noticing that $(x_1'x_2'-x_3')^2+x_1'^2+1 = (1+x_2'^2)x_1'^2-2x_1'x_2'x_3'+x_3'^2+1$; sending $x_1' \mapsto \frac{x_1'}{\sqrt{1+x_2'^2}}$ and $x_3' \mapsto x_3'\sqrt{1+x_2'^2}$ the expression becomes $(x_1'-x_2'x_3')^2+x_3'^2+1$ and lastly we send $x_1'-x_2'x_3' \mapsto x_1'\sqrt{1+x_3'^2}$:
\begin{align*}
	X'_{w_l} &= \int_{U(\R)} e\Biggl(v_1 \frac{\beta_1}{t_2} \frac{x_2'}{1+x_2'^2}+v_1 \frac{\beta_1}{t_2}\frac{x_1'x_3'}{(1+x_2'^2)\sqrt{1+x_3'^2}} \\
	& \qquad +v_2 \frac{\beta_2}{t_1} \frac{x_1'\sqrt{1+x_2'^2}}{(1+x_1'^2)\sqrt{1+x_3'^2}} \Biggr) \\
	& \qquad \e{t_1 \frac{x_2' x_3'}{\sqrt{1+x_2'^2}}+t_1 \frac{x_1'\sqrt{1+x_3'^2}}{\sqrt{1+x_2'^2}}+t_2 x_2'} \\
	& \qquad \paren{1+x_1'^2}^{\frac{-1-u_1+2u_2}{2}} \paren{1+x_2'^2}^{\frac{-1+2u_1-u_2}{2}} \paren{1+x_3'^2}^{\frac{-1+u_1+u_2}{2}} dx'.
\end{align*}

For each of the six terms in the exponential, we apply the inverse Mellin transform
\begin{align}
\label{eq:PsiThetaInvMellin}
	\e{x} &= \lim_{\theta\to\frac{\pi}{2}^-} \frac{1}{2\pi i} \int_{\Re(t) = c} \abs{2\pi x}^{-t} e^{it\theta \sgn(x)} \Gamma\paren{t} \, dt,
\end{align}
for $x \ne 0$ and $c > 0$, which follows from the definition of the gamma function and Mellin inversion.
We have the $X'_{w_l}$ function as the limit in $\theta$ of
\begin{align*}
	& X'_{w_l}(u,v,\beta,t,\theta) \\
	&= \frac{1}{(2\pi i)^6} \int_{\Re(r)=\nu} \int_{U(\R)} (4\pi^2 \beta_1)^{-r_1-r_2} (4\pi^2 \beta_2)^{-r_3} (2\pi t_1)^{r_3-r_4-r_5} (2\pi t_2)^{r_1+r_2-r_6} \\
	& \qquad \abs{x_1'}^{-r_2-r_3-r_5} \abs{x_2'}^{-r_1-r_4-r_6} \abs{x_3'}^{-r_2-r_4} \\
	& \qquad \exp -i\theta\paren{r_1 v_1 \sgn(x_2')+r_2 v_1 \sgn(x_1' x_3')+r_3 v_2 \sgn(x_1')} \\
	& \qquad \exp -i\theta\paren{r_4 \sgn(x_2' x_3')+r_5 \sgn(x_1')+r_6 \sgn(x_2')} \\
	& \qquad \paren{1+x_1'^2}^{\frac{-1-u_1+2u_2+2r_3}{2}} \paren{1+x_2'^2}^{\frac{-1+2u_1-u_2+2r_1+2r_2-r_3+r_4+r_5}{2}} \\
	& \qquad \paren{1+x_3'^2}^{\frac{-1+u_1+u_2+r_2+r_3-r_5}{2}} dx' \, \paren{\prod_{j=1}^6 \Gamma\paren{r_j}} \,dr.
	% 2\pi i from each of the 6 Mellin expansions of e_\theta %
\end{align*}
By absolute convergence, the limit in $\theta$ may be pulled outside the $J_{w_l, \mu}$ function, and even outside the sum of Kloosterman sums.

Collecting by sign gives
\begin{align*}
	X'_{w_l} &= \frac{4}{(2\pi i)^6} \int_{\Re(r)=\nu} (4\pi^2 \beta_1)^{-r_1-r_2} (4\pi^2 \beta_2)^{-r_3} (2\pi t_1)^{r_3-r_4-r_5} (2\pi t_2)^{r_1+r_2-r_6} \\
	& \Gamma\paren{r_1} \Gamma\paren{r_2} \Gamma\paren{r_3} \Gamma\paren{r_4} \Gamma\paren{r_5} \Gamma\paren{r_6} A_{w_l}'(r,v,\theta) \\
	& \int_{(\R^+)^3} x_1'^{-r_2-r_3-r_5} x_2'^{-r_1-r_4-r_6} x_3'^{-r_2-r_4} \paren{1+x_1'^2}^{\frac{-1-u_1+2u_2+2r_3}{2}} \\
	& \paren{1+x_2'^2}^{\frac{-1+2u_1-u_2+2r_1+2r_2-r_3+r_4+r_5}{2}} \paren{1+x_3'^2}^{\frac{-1+u_1+u_2+r_2+r_3-r_5}{2}} dx' \,dr,
	% 4 for A_{w_l}' %
\end{align*}
where
\begin{align*}
	A_{w_l}' =& \frac{1}{4} \sum_{\varepsilon_1,\varepsilon_2,\varepsilon_3\in\set{\pm 1}} \exp -i\theta (r_1 v_1 \varepsilon_2+r_2 v_1 \varepsilon_1 \varepsilon_3+r_3 v_2 \varepsilon_1 \\
	& \qquad +r_4 \varepsilon_2\varepsilon_3+r_5 \varepsilon_1+r_6 \varepsilon_2) \\
	=& \cos \theta(r_2 v_1+r_3 v_2+r_5) \cos \theta(r_1 v_1+r_4+r_6) \\
	& \qquad + \cos \theta(-r_2 v_1+r_3 v_2+r_5) \cos \theta(r_1 v_1-r_4+r_6),
\end{align*}
$\nu_1,\ldots,\nu_6 = \epsilon$ are small compared to $\Re(u_1-2u_2) > \frac{1}{2}$, $\Re(-2u_1+u_2) > \frac{1}{2}$, and $\Re(-u_1-u_2) > \frac{1}{2}$.

The inner integral may be evaluated by
\begin{align}
\label{eq:xSquaredPlusOneMellin}
	\int_0^\infty (1+x^2)^u x^t dx &= \frac{1}{2} B\paren{\frac{t+1}{2},\frac{-2u-t-1}{2}},
\end{align}
for $-1 < \Re(t) < -1-2\Re(u)$, which follows from the definition of the beta function $B(u,v) = \frac{\Gamma(u)\Gamma(v)}{\Gamma(u+v)}$.
This gives the Mellin-Barnes integral representation
\begin{align*}
	X'_{w_l} &= \frac{1}{2} \frac{1}{(2\pi i)^6} \int_{\Re(r)=\nu} (4\pi^2 \beta_1)^{-r_1-r_2} (4\pi^2 \beta_2)^{-r_3} (2\pi t_1)^{r_3-r_4-r_5} (2\pi t_2)^{r_1+r_2-r_6} \\
	& \qquad \Gamma\paren{r_1} \Gamma\paren{r_2} \Gamma\paren{r_3} \Gamma\paren{r_4} \Gamma\paren{r_5} \Gamma\paren{r_6} A_{w_l}'(r,v,\theta) \\
	& \qquad B\paren{\frac{1-r_2-r_3-r_5}{2},\frac{u_1-2u_2+r_2-r_3+r_5}{2}} \\
	& \qquad B\paren{\frac{1-r_1-r_4-r_6}{2},\frac{-2u_1+u_2-r_1-2r_2+r_3-r_5+r_6}{2}} \\
	& \qquad B\paren{\frac{1-r_2-r_4}{2},\frac{-u_1-u_2-r_3+r_4+r_5}{2}} \,dr,
\end{align*}
which converges absolutely for $0 < \theta < \frac{\pi}{2}$ because the exponential decay of the $\Gamma\paren{r_i}$ functions is not quite offset by the exponential growth of the $A_{w_l}'$ function.

Having a Mellin-Barnes integral for $X'_{w_l}$ means we may compute one for $T_{w_l}$ simply by pulling the $t$ integral inside and applying the Mellin transform of the Whittaker function:
\begin{align*}
	& T_{w_l}(v,\beta,\theta) = \\
	& \frac{2}{\pi^2} \frac{1}{(2\pi i)^6} \int_{\Re(r)=\nu} (\pi^2 \beta_1)^{-u_1-r_1-r_2} (\pi^2 \beta_2)^{-u_2-r_3} 2^{-r_1-r_2-r_3-r_4-r_5-r_6} \\
	& \qquad \Gamma\paren{r_1} \Gamma\paren{r_2} \Gamma\paren{r_3} \Gamma\paren{r_4} \Gamma\paren{r_5} \Gamma\paren{r_6} A_{w_l}'(r,v,\theta) \\
	& \qquad G^*\paren{\paren{2+2\Delta+u_2+r_3-r_4-r_5,1+\Delta+u_1+r_1+r_2-r_6}, -\mu} \\
	& \qquad B\paren{\frac{1-r_2-r_3-r_5}{2},\frac{u_1-2u_2+r_2-r_3+r_5}{2}} \\
	& \qquad B\paren{\frac{1-r_1-r_4-r_6}{2},\frac{-2u_1+u_2-r_1-2r_2+r_3-r_5+r_6}{2}} \\
	& \qquad B\paren{\frac{1-r_2-r_4}{2},\frac{-u_1-u_2-r_3+r_4+r_5}{2}} \,dr.
\end{align*}

Assuming absolute convergence at $\theta = \frac{\pi}{2}$ for some choice of the contours, we send $r \mapsto (s_1-u_1-t_1,t_1,s_2-u_2,t_1+t_2,1-r_1+s_2-t_1-t_2,1-r_2+s_1)$, so we may write
\begin{align}
\label{eq:TwlFinal}
	T_{w_l,\mu}(u, \alpha) =& \frac{1}{(2\pi i)^2} \int_{\Re(s)=\mathfrak{s}} \abs{4\pi^2 \alpha_1}^{-s_1} \abs{4\pi^2 \alpha_2}^{-s_2} \Gamma\paren{s_2-u_2} \\
	& \qquad N_{w_l}(s,\mu,\sgn(\alpha)) \,ds, \nonumber
\end{align}
\begin{align*}
	N_{w_l,1}(s,u,v) =& \frac{1}{2 \pi i} \int_{\Re(t_1)=\mathfrak{t}_1} \Gamma\paren{t_1} \Gamma\paren{s_1-u_1-t_1} R_{w_l}(s,u,v,t_1) \,dt_1, \\
	R_{w_l}(s,u,v,t_1) =& \frac{1}{(2 \pi i)^2} \int_{\Re(r)=\mathfrak{r}} G^*((1+2\Delta,\Delta)+r,-\mu) \Gamma\paren{1+s_1-r_2} \\
	& \qquad N_{w_l,2}(r,s,u,v,t_1) \,dr,
\end{align*}
\begin{align*}
	N_{w_l,2}(r,s,u,v,t_1) &= \frac{2^{r_1+r_2+u_1+u_2}}{2 \pi^2 (2\pi i)} \int_{\Re(t_2)=\mathfrak{t}_2} A_{w_l}(r,s,t,u,v) \\
	& \qquad \Gamma\paren{t_1+t_2} \Gamma\paren{1-r_1+s_2-t_1-t_2} \\
	& \qquad B\paren{\frac{r_1-2s_2+t_2+u_2}{2},\frac{1-r_1-t_2+u_1-u_2}{2}} \\
	& \qquad B\paren{\frac{r_2-2s_1-t_2+u_1}{2},\frac{r_1-r_2+t_2-u_1}{2}} \\
	& \qquad B\paren{\frac{1-2t_1-t_2}{2},\frac{1-r_1-u_1}{2}} \,dt_2,
\end{align*}
\begin{align*}
	A_{w_l} =& \sin \frac{\pi}{2}\paren{v_1 t_1+v_2(s_2-u_2)-r_1+s_2-t_1-t_2} \\
	& \qquad \qquad \sin \frac{\pi}{2}\paren{v_1(s_1-t_1-u_1)+t_1+t_2-r_2+s_1} \\
	& \qquad + \sin \frac{\pi}{2}\paren{-v_1 t_1+v_2(s_2-u_2)-r_1+s_2-t_1-t_2} \\
	& \qquad \qquad \sin \frac{\pi}{2}\paren{v_1(s_1-t_1-u_1)-t_1-t_2-r_2+s_1}.
\end{align*}
We leave the choice of contours and discussion of absolute convergence to sections \ref{sec:SumsOfKloostermanSums} and \ref{sec:Bounds}, but we should mention here that before taking the limit in $\theta$, we may freely shift contours -- avoiding the poles of the gamma functions -- again because the exponential decay of the gamma functions is not quite offset by the exponential growth of the $A_{w_l}$ function.

\section{Applications}
\label{sec:Applications}

\subsection{\texorpdfstring{Asymptotics of the $J_{w_l,\mu}$ Function}{Asymptotics of the Jwl,mu Function}}
\label{sec:JwlAsymptotic}
We want to achieve the highest power of the $\beta$ variables possible -- this gives the fastest convergence of the Kloosterman zeta function, so we want to move the $s$ variables as negative as possible.
As the $s$ variables are indirectly bounded below by $(\mu_1, -\mu_2)$, any terms which allow us to cross below those lines will be considered small.
Thus we only care about the $u = (\mu_1,-\mu_2)$ residue.
Then we shift the $s$ integrals back, with poles at $s_1 = \mu_1+t_1$ and $s_2 = -\mu_2$, and we shift the $t_1$ integral back, with a pole at $t_1 = 0$.
So far, we have
\begin{align}
\label{eq:JwlAsympViaRwl}
	J_{w_l,\mu}(y) \sim& \frac{3}{8 \pi^5} \abs{4\pi^2 y_1}^{-\mu_1} \abs{4 \pi^2 y_2}^{\mu_2} G^*_b(\mu) \\
	& \qquad \frac{k_\text{adj}(\mu)}{\abs{c_3(\mu)}^2} R_{w_l}((\mu_1,-\mu_2),(\mu_1,-\mu_2),v,0), \nonumber
\end{align}
as $y \to 0$.
This yields the error terms of \propref{prop:JwlAsymptotic}:
\begin{align}
\label{eq:Ewl1}
	E_{w_l,1}(\mu,y) =& \frac{-3i}{128 \pi^6} \frac{k_\text{adj}(\mu)}{\abs{c_3(\mu)}^2} \int_{\Re(u_1) = -\frac{1}{2}-10\epsilon} G^*_r(u_1,\mu) T_w((u_1,-\mu_2),\alpha) du_1 \\
\label{eq:Ewl2}
	E_{w_l,2}(\mu,y) =& \frac{-3i}{128 \pi^6} \frac{k_\text{adj}(\mu)}{\abs{c_3(\mu)}^2} \int_{\Re(u_2) = -\frac{1}{2}-10\epsilon} G^*_l(u_2,\mu) T_w((\mu_1,u_2),\alpha) du_2 \\
\label{eq:Ewl3}
	E_{w_l,3}(\mu,y) =& \frac{-3}{2^8 3 \pi^7} \frac{k_\text{adj}(\mu)}{\abs{c_3(\mu)}^2} \int_{\Re(u) = -\frac{1}{2}-10\epsilon} G^*(u,\mu) T_w(u,\alpha) du
\end{align}
\begin{align}
\label{eq:Ewl4}
	E_{w_l,4}(\mu,y) =& \frac{3}{32 \pi^5 (2\pi i)} \frac{6}{2\pi i} G^*_b(\mu) \frac{k_\text{adj}(\mu)}{\abs{c_3(\mu)}^2} \int_{\Re(s_1)=\Re(\mu_1)-\epsilon} \abs{4\pi^2 y_1}^{-s_1} \\
	& \qquad \abs{4\pi^2 y_2}^{\mu_2} N_{w_l,1}((s_1,\mu_2),(\mu_1,-\mu_2),\sgn(y)) \,ds_2, \nonumber
\end{align}
\begin{align}
\label{eq:Ewl5}
	E_{w_l,5}(\mu,y) =& \frac{3}{32 \pi^5 (2\pi i)^2} G^*_b(\mu) \frac{k_\text{adj}(\mu)}{\abs{c_3(\mu)}^2} \int_{\Re(s_2)=\Re(-\mu_2)-\epsilon} \abs{4\pi^2 y_2}^{-s_2} \\
	& \qquad \Gamma\paren{s_2+\mu_2} \int_{\Re(t_1)=\mathfrak{t}_1} \abs{4\pi^2 y_1}^{-\mu_1-t_1} \Gamma\paren{t_1} \nonumber \\
	& \qquad R_{w_l}((\mu_1+t_1,s_2),(\mu_1,-\mu_2),v,t_1) \,dt_1 \,ds_1, \nonumber
\end{align}
\begin{align}
\label{eq:Ewl6}
	E_{w_l,6}(\mu,y) =& \frac{3}{32 \pi^5 (2\pi i)^2} G^*_b(\mu) \frac{k_\text{adj}(\mu)}{\abs{c_3(\mu)}^2} \int_{\Re(s)=\Re(\mu_1,-\mu_2)-\epsilon} \abs{4\pi^2 y_1}^{-s_1} \\
	& \qquad \abs{4\pi^2 y_2}^{-s_2} \Gamma\paren{s_1-\mu_1} N_{w_l,1}(s,(\mu_1,-\mu_2),\sgn(y)) \,ds, \nonumber
\end{align}
\begin{align}
\label{eq:Ewl7}
	E_{w_l,7}(\mu,y) =& \frac{3}{32 \pi^5 (2\pi i)} G^*_b(\mu) \frac{k_\text{adj}(\mu)}{\abs{c_3(\mu)}^2} \int_{\Re(t_1)=-\epsilon} \abs{4\pi^2 y_1}^{-\mu_1} \abs{4\pi^2 y_2}^{\mu_2-t_1} \\
	& \qquad \Gamma\paren{t_1} R_{w_l}((\mu_1+t_1,-\mu_2),(\mu_1,-\mu_2),v,t_1) \,dt_1. \nonumber
\end{align}

Returning to \eqref{eq:JwDef}, \eqref{eq:TwFinal} and \eqref{eq:XwlFinal}, we may compute the main term explicitly.
For $\Re(\mu) = \paren{-\frac{1}{2}-\epsilon,\frac{1}{2}+\epsilon}$,
\begin{align*}
	& \lim_{y\to 0} \abs{\pi y_1}^{\mu_1} \abs{\pi y_2}^{-\mu_2} J_{w_l,\mu}(y) \\
	&= \frac{3}{32 \pi^5} G^*_b(\mu) \frac{k_\text{adj}(\mu)}{\abs{c_3(\mu)}^2} \int_{Y(\R)} W(t,-\mu,\psi_{11}) \\
	& \qquad \lim_{y\to 0} X'_{w_l}((\mu_1,-\mu_2),\sgn(y),\abs{y},t) t_1^{3+2\Delta-\mu_2} (\pi t_2)^{2+\Delta+\mu_1} \, dt,
\end{align*}
by dominated convergence.
The limit in $X'_{w_l}$ is actually a Whittaker function,
\begin{align*}
	\lim_{y\to 0} X'_{w_l}((\mu_1,-\mu_2),\sgn(y),\abs{y},t) =& \int_{U(\R)} \wbar{\psi_{t}}(x') {y_1^*}^{1-\mu_1} {y_2^*}^{1+\mu_2} dx' \\
	=& W(I, (\mu_2,\mu_3,\mu_1), \psi_{t}) \\
	=& t_1^{-1+\mu_2} t_2^{-1-\mu_1} W(t, (\mu_2,\mu_3,\mu_1), \psi_{11}),
\end{align*}
again, by dominated convergence.
Applying this to the limit of $J_{w_l,\mu}$ gives
\begin{align*}
	& \lim_{y\to 0} \abs{\pi y_1}^{\mu_1} \abs{\pi y_2}^{-\mu_2} J_{w_l,\mu}(y) \\
	&= \frac{3}{32 \pi^5} G^*_b(\mu) \frac{k_\text{adj}(\mu)}{\abs{c_3(\mu)}^2} \frac{1}{\Lambda(-\mu)\Lambda(\mu_2,\mu_3,\mu_1)} \\
	& \qquad \int_{Y(\R)} W^*(t,-\mu,\psi_{11}) W^*(t, \mu, \psi_{11}) t_1^{2+2\Delta} t_2^{1+\Delta} \, dt \\
	&= \frac{3}{128 \pi^{11}} G^*_b(\mu) \frac{k_\text{adj}(\mu)}{\abs{c_3(\mu)}^2} \frac{\prod_{j<k} \Gamma\paren{\frac{1+\Delta+\mu_k-\mu_j}{2}} \Gamma\paren{\frac{1+\Delta+\mu_j-\mu_k}{2}}}{\Lambda(-\mu)\Lambda(\mu_2,\mu_3,\mu_1)}
	% \frac{\Gamma(1)^3}{4 \pi^{6} \Gamma\paren{3}} \frac{\prod_{j<k} \Gamma\paren{\frac{2+\mu_k-\mu_j}{2}} \Gamma\paren{\frac{2+\mu_j-\mu_k}{2}} from Stade's formula.
\end{align*}
thus $J_{w_l,\mu}(y) \sim \abs{\pi y_1}^{-\mu_1} \abs{\pi y_2}^{\mu_2} K_{w_l}(\mu)$.
This expression then agrees with right hand side of \eqref{eq:JwlAsympViaRwl} over the entire range of holomorphy by analytic continuation and we have \propref{prop:JwlAsymptotic}.
Note that we induced an asymmetry in the definition of the $J_{w_l,\mu}$ function, hence the asymmetry here; this is a subtle but important point as it allows us to avoid some symmetry requirements for the test functions of \thmref{thm:PartialInversion} and \thmref{thm:SumsOfKloostermanSums}.

\subsection{Partial Inversion Formula}
\label{sec:PartialInversion}
If we take our test function to be \eqref{eq:PartialInversionkhat} then in $H_{w_l}$, we move $\Re(q) \mapsto \Re(\mu_1,-\mu_2)+\epsilon$, and apply the asymptotics of $J_{w_l, \mu}$ at the double residue $q=(\mu_1,-\mu_2)$ gives \thmref{thm:PartialInversion} with
\begin{align}
\label{eq:F1}
	F_1(\hat{f}; y) =& \frac{1}{(2 \pi i)^2} \int_{\Re(\mu)=\eta} \hat{k}_{(\eta_1,-\eta_2)+\epsilon}(\mu) J_{w_l,\mu}(y) \, d\mu,
\end{align}
\begin{align}
\label{eq:F2}
	& F_2(\hat{f}; y) = \frac{1}{(2 \pi i)^2} \int_{\Re(\mu)=\eta} \int_{\Re(q_2) = -\eta_2+\epsilon} \hat{f}(\mu_1,q_2) \frac{J_{w_l,\mu}(y)}{K_{w_l}(\mu_1,-q_2)} \\
	& k_\text{conv}(\mu,(\mu_1,-q_2)) \frac{\paren{q_2+\mu_1}\paren{2\mu_1-q_2}\paren{2q_2-\mu_1}}{\prod_{j\ne 1}\paren{\mu_1-\mu_j}\paren{q_2+\mu_j}} \, dq_2 \, d\mu, \nonumber
\end{align}
\begin{align}
\label{eq:F3}
	& F_3(\hat{f}; y) = \frac{1}{(2 \pi i)^2} \int_{\Re(\mu)=\eta} \int_{\Re(q_1) = \eta_1+\epsilon} \hat{f}(q_1,-\mu_2) \frac{J_{w_l,\mu}(y)}{K_{w_l}(q_1,\mu_2)} \\
	& k_\text{conv}(\mu,(q_1,\mu_2)) \frac{\paren{q_1-\mu_2}\paren{2q_1+\mu_2}\paren{2\mu_2-q_1}}{\prod_{j\ne 2}\paren{q_1-\mu_j}\paren{\mu_j-\mu_2}} \, dq_1 \, d\mu. \nonumber
\end{align}
\begin{align}
\label{eq:Fjplus3}
	F_{j+3}(\hat{f}; y) =& \frac{1}{(2 \pi i)^2} \int_{\Re(\mu) = \eta} \hat{f}(\mu_1,-\mu_2) \frac{E_{w_l,j}(\mu,y)}{K_{w_l}(\mu)} d\mu.
\end{align}

It may be possible to study the Kloosterman zeta functions directly by simply not integrating over $q$ in $\hat{k}$; this would require a test function $\hat{k}$ which cancels the intermediate terms in $H_{w_l}$ and $J_{w_l}$ (the terms with a residue at one of $q_1$ or $q_2$, but not both, and the terms with a residue at one of $s_1$ or $s_2$, but not both, and the term with a residue in $t_1$).

\subsection{Sums of Kloosterman Sums}
\label{sec:SumsOfKloostermanSums}
We are now ready for the proof of \thmref{thm:SumsOfKloostermanSums}.
Let $g(y) = f(X y_1,Y y_2)$, then the assumption that $f$ have compact support is not strictly necessary, we merely need holomorphy of $\hat{g}$ on $\Re(q_1),\Re(q_2)\in \paren{-\frac{1}{2}-\epsilon,-\epsilon}$ and the bound
\[ \hat{g}(q) \ll \frac{X^{-\Re(q_1)}}{\abs{q_1}^8}\frac{Y^{-\Re(q_2)}}{\abs{q_2}^8} \ll X^{-\Re(q_1)} Y^{-\Re(q_2)} \abs{q_1 q_2 (q_2-q_1)}^{-4}, \]
which follows by integration by parts eight times in each $y$ variable.
\thmref{thm:SumsOfKloostermanSums} follows from \thmref{thm:PartialInversion} by fixing the contours of the error terms and those of the cusp form terms, Eisenstein series terms, and non-long-element Kloosterman sum terms and justifying their absolute convergence in the new locations.
Specifically, we want to shift the contours in $q$ as far to the right as possible.

For the cusp form terms in \eqref{eq:LisKuznetsovFormula}, we may shift the $q$ contours of $\hat{k}$ up to $\mathfrak{q}=-\frac{5}{14}-\epsilon$ without encountering poles at any of the $q_1-\mu_i$ or $q_2+\mu_i$ terms, thanks to the Kim-Sarnak result.
The $K_{w_l}(q_1,-q_2)^{-1}$ term has poles at $-q_1-q_2=0$, $-2q_1+q_2=0$ and $q_1-2q_2=0$, but we need not encounter these and they are cancelled by the terms in the numerator as well.
Now the mean-value estimates of \thmref{thm:MeanValueEstimates} show that the sum over the cusp forms converges absolutely so we have the bound $(XY)^{\frac{5}{14}+\epsilon}$ here.

The terms in \eqref{eq:LisKuznetsovFormula} for both types of Eisenstein series have $\hat{k}$ evaluated at $\Re(\mu) = 0$, as does the trivial Weyl element term, and the sums of Kloosterman sums at the $w_4$ and $w_5$ Weyl elements still converge absolutely at $\Re(\mu)=(-\epsilon,\epsilon)$ so for each of these terms we may shift the $q$ contours to $\mathfrak{q}=-2\epsilon$.
Again, absolute convergence gives $(XY)^{2\epsilon}$.

In the section on bounds, we will show the $F_j$ error terms of \thmref{thm:PartialInversion} are all bounded by $(XY)^{10\epsilon}(X^{\frac{1}{2}}+Y^{\frac{1}{2}})$ by taking the contours as in \tableref{tab:FjContours}, with $\mathfrak{s} = -\frac{1}{2}-\epsilon$.
Our choice of $\mathfrak{s}$ maintains the absolute convergence of the Kloosterman zeta functions and the exponent on the bounds come from $\mathfrak{q}$.
Lastly, we use $\Re(\mu)=\paren{-2\epsilon,\epsilon}$ for the $w_4$ term and $\eta=\paren{-\epsilon,2\epsilon}$ for the $w_5$ term.
The choice of contours for the $w_4$ and $w_5$ terms are for convenience and to reduce the number of derivatives necessary in \thmref{thm:SumsOfKloostermanSums}.

\begin{landscape}
\begin{table}
\begin{tabular}{rcccccc}
	j & $\mathfrak{q}$ & $\eta$ & $\mathfrak{u}$ & $\mathfrak{r}$ & $\mathfrak{t}$ \\
	\hline
	1 & $-\epsilon$ & $\paren{-\frac{1}{2}-3\epsilon,\frac{1}{2}+3\epsilon}$ & -- & -- & -- \\

	2 & $\paren{\mathdash,-\epsilon}$ & $\paren{-\frac{1}{2}-3\epsilon,\frac{1}{2}+3\epsilon}$ & -- & -- & -- \\

	3 & $\paren{-\epsilon,\mathdash}$ & $\paren{-\frac{1}{2}-3\epsilon,\frac{1}{2}+3\epsilon}$ & -- & -- & -- \\

	4 & -- & $\paren{-3\epsilon,\frac{1}{2}+3\epsilon}$ & $\paren{-\frac{1}{2}-3\epsilon,\mathdash}$ & $\paren{\frac{1}{2}-4\epsilon,100\epsilon}$ & $\epsilon$ \\

	5 & -- & $\paren{-\frac{1}{2}-3\epsilon,3\epsilon}$ & $\paren{\mathdash,-\frac{1}{2}-3\epsilon}$ & $\paren{\frac{1}{2}-4\epsilon,100\epsilon}$ & $\epsilon$ \\

	6 & -- & $(-3\epsilon,3\epsilon)$ & $-\frac{1}{2}-3\epsilon$ & $\paren{\frac{1}{2}-4\epsilon,100\epsilon}$ & $\epsilon$ \\

	7 & -- & $\paren{-\epsilon,\frac{1}{2}+3\epsilon}$ & -- & $\paren{-4\epsilon,100\epsilon}$ & $\paren{\epsilon,\frac{1}{2}-\epsilon}$ \\

	8 & -- & $\paren{-\frac{1}{2}-3\epsilon,\epsilon}$ & -- & $\paren{-3\epsilon,7\epsilon}$ & $\paren{6\epsilon,\frac{1}{2}-6\epsilon}$ \\

	9 & -- & $(-\epsilon,\epsilon)$ & -- & $0$ & $\epsilon$ \\

	10 & -- & $\paren{-\epsilon,\frac{1}{2}+3\epsilon}$ & -- & $\paren{\epsilon,4\epsilon}$ & $\paren{-\frac{1}{2}+5\epsilon,1-10\epsilon}$
\end{tabular}
\caption{Contours for the $F_j$ error terms.}
\label{tab:FjContours}
\end{table}
\end{landscape}

\section{Bounds for the Mellin-Barnes Integrals}
\label{sec:Bounds}
We are left with two items to prove, which are essentially the same:
First, completing \thmref{thm:KuznetsovSimplification} requires justifying the growth hypothesis on $\hat{k}$, in other words, bounding $J_{w,\mu}$, which is also desirable for \thmref{thm:MeanValueEstimates}.
Second, the evaluation of the integral transforms, the asymptotics of $E_{w_l,j}$ and $F_j$ given in \propref{prop:JwlAsymptotic} and \thmref{thm:PartialInversion}, and the completion of \thmref{thm:SumsOfKloostermanSums} all require absolute convergence of the Mellin-Barnes integrals.

It is difficult to obtain a general bound for $T_{w_l}$ and $J_{w_l,\mu}$ that works for all ranges of the $\eta$ and $\mathfrak{s}$ parameters, hence it is also difficult to show that these functions converge absolutely over any range of $\Re(\mu)$.
Therefore, we will not actually show that these functions are holomorphic over any such range.
This leads one to question whether it is valid to shift contours as we have freely done; for the skeptical reader, we have a simple justification:
Do the shifting before taking the limit in $\theta$ back in the original construction.
As we have the bound $A_{w_l}' \prod_{j=1}^6 \Gamma(r_j) \ll \prod_{j=1}^6 \abs{r_j}^{\Re(r_j)-\frac{1}{2}} \exp\paren{\paren{\theta-\frac{\pi}{2}}\abs{\Im(r_j)}}$, both convergence and the validity of the shifts are obvious.
Then we only require that the end product converges absolutely at $\theta=\frac{\pi}{2}$, and that is what we will show.

The fundamental asymptotic here is Stirling's formula:
For $\Re(z)$ in a compact subset of $\R$ (not containing a pole of the gamma function),
\begin{align*}
	\abs{\Gamma(z)} &\sim \sqrt{2\pi} \abs{z}^{\Re(z)-\frac{1}{2}} e^{-\frac{\pi}{2}\abs{\Im(z)}},
\end{align*}
which leads us to integrals of products in the form
\begin{align}
\label{eq:GeneralProductInt}
	\int_{\Re(u)=\mathfrak{u}} \prod_i \abs{a_{i,1} u_1 + \ldots + a_{i,n} u_n + b_{i,1} v_1 + \ldots + b_{i,m} v_m}^{c_i} \, \abs{du},
\end{align}
where $a_{i,j},\mathfrak{u}_i,c_i \in \R$ are fixed, with $v_i \in \C$ having fixed real part.
Note that for $a$ and $c$ non-zero and fixed, $b\in\R$, we have $\abs{a+bi} \asymp \abs{c+bi}$.
Provided the exponents are not somehow accumulating on any subspace, we would expect such an integral to converge when $\sum_i c_i < -n-1$, and we give a series of lemmas designed to show that these converge in our situation.

Bounds for integrals of the above type are derived from H\"older's inequality and the following lemma:
\begin{lem}
\label{lem:PairConvLemma}
	Suppose $a_1+a_2 < -1$ with $a_1$ and $a_2$ fixed, and $s \ge 0$, then
	\[ \int_{-\infty}^\infty \abs{1+i(s+t)}^{a_1} \abs{1+i(s-t)}^{a_2} dt \ll \abs{1+is}^{\Max{a_1,a_2,a_1+a_2+1}}. \]
\end{lem}
\begin{proof}
We will make repeated use of the integrals
\[ \int_0^u \abs{1+it}^\alpha dt = u \, \pFq{2}{1}{\frac{1}{2}, -\frac{\alpha}{2}}{\frac{3}{2}}{-u^2} \ll \abs{1+iu}^{\Max{0,\alpha+1}} \]
and
\[ \int_u^\infty \abs{1+it}^\alpha dt \le u^{\alpha+1} \int_1^\infty t^\alpha dt \ll \abs{1+iu}^{\alpha+1}, \qquad \Re(\alpha) < -1, \]
for $u \ge 1$.

The result is obvious if $s < 1$, so we assume $s \ge 1$, and split the integral at $-2s$,$0$, and $2s$, call the resulting integrals $I_1$, $I_2$, $I_3$, and $I_4$, say.
For the first integral, we substitute $t \mapsto -t-2s$, so it becomes $I_1= I_{1,a} + I_{1,b}$:
\begin{align*}
	I_{1,a} =& \int_0^s \abs{1+i(s+t)}^{a_1} \abs{1+i(t+3s)}^{a_2} dt \\
	\asymp& \abs{1+is}^{a_1+a_2} \int_0^s dt \\
	\ll& \abs{1+is}^{a_1+a_2+1}, \\
	I_{1,b} =& \int_s^\infty \abs{1+i(s+t)}^{a_1} \abs{1+i(t+3s)}^{a_2} dt \\
	\asymp& \int_s^\infty \abs{1+it}^{a_1+a_2} dt \\
	\ll& \abs{1+is}^{a_1+a_2+1},
\end{align*}
and similarly for the remaining terms.
\end{proof}

We will occasionally encounter positive exponents in the integrals of type \eqref{eq:GeneralProductInt}, but thankfully these always occur in the terms coming from the beta functions, so we have
\begin{lem}
\label{lem:BetaConvLemma}
	Suppose $\Re(v) = \mathfrak{v}$ with $\mathfrak{v}_1 + \mathfrak{u}$, $\mathfrak{v}_2 - \mathfrak{u}$ not non-positive integers, and $\mathfrak{u}$, $\mathfrak{v}$ fixed, and $p > 0$, then
	\[ \paren{\int_{\Re(u) = \mathfrak{u}} \abs{B(v_1+u,v_2-u)}^p \abs{du}}^{\frac{1}{p}} \ll \abs{v_1+v_2}^{\Max{\mathfrak{u}-\mathfrak{v}_2,-\mathfrak{u}-\mathfrak{v}_1,\frac{1}{p}-\frac{1}{2}}}. \]
\end{lem}
Note that this no longer requires $\Re(v_1+v_2-1) < -1$ as it would if we applied the previous lemma; this is because we are using the exponential decay of the gamma functions.
\begin{proof}
For $\Re(u)$, $\Re(v_1)$, $\Re(v_2)$ fixed, applying the second form of Stirling's formula shows that $B(u-v_1,v_2-u)$ decays exponentially in $u$ unless
\[ \Max{\Im(v_1),\Im(v_2)}>\Im(u)>\Min{\Im(v_1),\Im(v_2)}, \]
and in that case, the exponential parts cancel, so the proof is precisely the same as \lemref{lem:PairConvLemma}, without the equivalent of requiring $a_1+a_2 < -1$.
\end{proof}

\subsection{\texorpdfstring{Bounds for the $E_{w_l, j}$}{Bounds for the Ewl,j}}
The general method of bounding the $E_{w_l,j}$ integrals is to pull the $s$ and $t$ integrals inside the $r$ integrals and apply H\"older's inequality repeatedly, using \lemref{lem:PairConvLemma} and using \lemref{lem:BetaConvLemma} to handle any positive exponents that occur in the polynomial parts of the beta functions.
Through this process we arrive at
\begin{prop}
	For the contours given in \tableref{tab:FjContours} and small $\Delta \gg \epsilon$, we have absolute convergence of all of the weight functions with
	\begin{align*}
		\abs{E_{w_l,1}(\mu,y)} \ll& \abs{y_1}^{\frac{1}{2}+\epsilon} \abs{y_2}^{\frac{1}{2}+\epsilon} M_\text{sym}\paren{2,0;\mu} \\
		\abs{E_{w_l,2}(\mu,y)} \ll& \abs{y_1}^{\frac{1}{2}+\epsilon} \abs{y_2}^{\frac{1}{2}+\epsilon} M_\text{sym}\paren{2,0;\mu} \\
		\abs{E_{w_l,3}(\mu,y)} \ll& \abs{y_1}^{\frac{1}{2}+\epsilon} \abs{y_2}^{\frac{1}{2}+\epsilon} M_\text{sym}\paren{\frac{5}{2},\frac{1}{2};\mu} \\
		\abs{E_{w_l,4}(\mu,y)} \ll& \abs{y_1}^{\frac{1}{2}+\epsilon} \abs{y_2}^{\frac{1}{2}+\epsilon} M_\text{sym}\paren{\frac{3}{2},0;\mu} \\
		\abs{E_{w_l,5}(\mu,y)} \ll& \abs{y_1}^{\frac{1}{2}+\epsilon} \abs{y_2}^{\frac{1}{2}+\epsilon} M_\text{sym}\paren{\frac{3}{2},0;\mu} \\
		\abs{E_{w_l,6}(\mu,y)} \ll& \abs{y_1}^{\frac{1}{2}+\epsilon} \abs{y_2}^{\frac{1}{2}+\epsilon} M_\text{sym}\paren{2,\frac{1}{2};\mu} \\
		\abs{E_{w_l,7}(\mu,y)} \ll& \abs{y_1}^{\frac{1}{2}+\epsilon} \abs{y_2}^{\frac{1}{2}+\epsilon} M_\text{sym}\paren{\frac{3}{2},0;\mu}.
	\end{align*}
\label{prop:Bounds}
\end{prop}

The complete proof is quite tedious, but, as an example, the bound on $E_{w_l,3}$ requires convergence of
\begin{align*}
	& \abs{M_{w_l,1}} = \\
	& \abs{y_1}^{-\mathfrak{s}_1} \abs{y_2}^{-\mathfrak{s}_2} \int_{\Re(s) = \mathfrak{s}} \int_{\Re(t)=\mathfrak{t}} \abs{s_1-u_1-t_1}^{-\frac{1}{2}+2\epsilon} \abs{s_2-u_2}^{-\frac{1}{2}+3\epsilon} \\
	& \qquad \abs{1+s_1-r_2}^{-101\epsilon} \abs{1+s_2-r_1-t_1-t_2}^{-\frac{1}{2}+\epsilon} \abs{t_1}^{-\frac{1}{2}+\epsilon} \abs{t_1+t_2}^{-\frac{1}{2}+2\epsilon} \\
	& \qquad \BetaFun{-\epsilon}{-\frac{1}{4}+2\epsilon}{-\frac{1}{4}+\epsilon}{u_2-2s_2+r_1+t_2}{1+u_1-u_2-r_1-t_2} \\
	& \qquad \BetaFun{-\frac{1}{4}+49\epsilon}{-49\epsilon}{-\frac{1}{4}+\epsilon}{u_1-2s_1+r_2-t_2}{-u_1+r_1-r_2+t_2} \\
	& \qquad \BetaFun{-\epsilon}{4\epsilon}{-\frac{1}{2}-2\epsilon}{1-2t_1-t_2}{1-u_1-r_1} \, \abs{dt} \, \abs{ds}
\end{align*}
for $\Re(u) = -\frac{1}{2}-4\epsilon+\delta$, $\Re(r) = \paren{\frac{1}{2}-4\epsilon,100\epsilon}$, $\mathfrak{s} = -\frac{1}{2}-\epsilon$, $\mathfrak{t} = \epsilon$.
For brevity, we have introduced the shorthand
\[ \BetaFun{a}{b}{c}{u}{v} := \abs{u}^a \abs{v}^b \abs{u+v}^c \exp-\frac{\pi}{4}\paren{\abs{\Im(u)}+\abs{\Im(v)}-\abs{\Im(u+v)}}, \]
and will supress terms $\BigO{\epsilon^2}$.

First, applying H\"older in $s_1$ with exponents $2-204\epsilon = \frac{1}{\frac{1}{2}+51\epsilon} +\BigO{\epsilon^2}$ and $2+204\epsilon = \frac{1}{\frac{1}{2}-51\epsilon} +\BigO{\epsilon^2}$ and \lemref{lem:PairConvLemma} pairing the first two occurrences of $s_1$ and the two terms in the second beta function, giving
\begin{align*}
	& \abs{M_{w_l,1}} \ll \\
	& \abs{y_1}^{-\mathfrak{s}_1} \abs{y_2}^{-\mathfrak{s}_2} \int_{\Re(s_2) = \mathfrak{s}_2} \int_{\Re(t)=\mathfrak{t}} \abs{1+u_1-r_2+t_1}^{-48\epsilon} \abs{s_2-u_2}^{-\frac{1}{2}+3\epsilon} \\
	& \qquad \abs{1+s_2-r_1-t_1-t_2}^{-\frac{1}{2}+\epsilon} \abs{t_1}^{-\frac{1}{2}+\epsilon} \abs{t_1+t_2}^{-\frac{1}{2}+2\epsilon} \abs{u_1-r_1+r_2-t_2}^{-51\epsilon} \\
	& \qquad \BetaFun{-\epsilon}{-\frac{1}{4}+2\epsilon}{-\frac{1}{4}+\epsilon}{u_2-2s_2+r_1+t_2}{1+u_1-u_2-r_1-t_2} \\
	& \qquad \BetaFun{-\epsilon}{4\epsilon}{-\frac{1}{2}-2\epsilon}{1-2t_1-t_2}{1-u_1-r_1} \, \abs{dt} \, \abs{ds_2}.
\end{align*}

Now repeat in $s_2$ with exponents $\frac{4}{3}-4\epsilon$ and $4+36\epsilon$, and with similar parings, giving
\begin{align*}
	& \abs{M_{w_l,1}} \ll \\
	& \abs{y_1}^{-\mathfrak{s}_1} \abs{y_2}^{-\mathfrak{s}_2} \int_{\Re(t)=\mathfrak{t}} \abs{1+u_1-r_2+t_1}^{-48\epsilon} \abs{1+u_1-u_2-r_1-t_2}^{-\frac{1}{4}+\epsilon} \\
	& \qquad \abs{1+u_2-r_1-t_1-t_2}^{-\frac{1}{4}+10\epsilon} \\
	& \qquad \abs{t_1}^{-\frac{1}{2}+\epsilon} \abs{t_1+t_2}^{-\frac{1}{2}+2\epsilon} \abs{u_1-r_1+r_2-t_2}^{-51\epsilon}  \\
	& \qquad \BetaFun{-\epsilon}{4\epsilon}{-\frac{1}{2}-2\epsilon}{1-2t_1-t_2}{1-u_1-r_1} \, \abs{dt}.
\end{align*}

Again, H\"older in $t_2$ with exponents $2+48\epsilon$, $\frac{1}{52\epsilon}$, and $2+160\epsilon$, pairing the first two occurrences and pairing each of the second two with the denominator of the beta function, giving
\begin{align*}
	& \abs{M_{w_l,1}} \ll \\
	& \abs{y_1}^{-\mathfrak{s}_1} \abs{y_2}^{-\mathfrak{s}_2} \abs{1-u_1-r_1}^{4\epsilon} \int_{\Re(t_1)=\mathfrak{t}_1} \abs{1+u_1-r_2+t_1}^{-48\epsilon} \abs{u_1-2u_2+t_1}^{-\epsilon} \\
	& \qquad \abs{2-2u_1-r_2-2t_1}^{-51\epsilon} \abs{t_1}^{-\frac{1}{2}+\epsilon} \abs{2-u_1-r_1-t_1}^{-\frac{1}{2}+49\epsilon} \, \abs{dt_1}.
\end{align*}

Lastly, we apply H\"older in $t_1$ with exponents $\frac{1}{46\epsilon}$, $\frac{1}{50\epsilon}$, and $1+96\epsilon$, pairing only the last two occurrences, giving
\begin{align*}
	& \abs{M_{w_l,1}} \ll \abs{y_1}^{-\mathfrak{s}_1} \abs{y_2}^{-\mathfrak{s}_2} \abs{1-u_1-r_1}^{4\epsilon} \abs{2-u_1-r_1}^{-46\epsilon} \ll \abs{y_1}^{-\mathfrak{s}_1} \abs{y_2}^{-\mathfrak{s}_2}.
\end{align*}

The remaining $E_{w_l,j}$ are similar; in each case, the $s$ and $t$ integrals converge absolutely and are bounded by $\abs{y_1}^{-\mathfrak{s}_1} \abs{y_2}^{-\mathfrak{s}_2}$ -- we ignore any extra convergence, so the $u$ and $r$ integrals separate, leaving us with integrals of the form
\begin{align*}
	& \frac{k_\text{adj}(\mu)}{\abs{c_3(\mu)}^2} \int_{\Re(u)=\mathfrak{u}} \abs{G^*(u,\mu)} \abs{du} \int_{\Re(r)=\mathfrak{r}} \abs{G^*((1+2\Delta,\Delta)+r,-\mu)} \abs{dr}, \\
	& \frac{k_\text{adj}(\mu)}{\abs{c_3(\mu)}^2} \int_{\Re(u_2)=\mathfrak{u}_2} \abs{G^*_l(u_2,\mu)} \abs{du_2} \int_{\Re(r)=\mathfrak{r}} \abs{G^*((1+2\Delta,\Delta)+r,-\mu)} \abs{dr}, \\
	& \frac{k_\text{adj}(\mu)}{\abs{c_3(\mu)}^2} \int_{\Re(u_1)=\mathfrak{u}_1} \abs{G^*_r(u_1,\mu)} \abs{du_1} \int_{\Re(r)=\mathfrak{r}} \abs{G^*((1+2\Delta,\Delta)+r,-\mu)} \abs{dr}, \\
	& \frac{k_\text{adj}(\mu)}{\abs{c_3(\mu)}^2} G^*_b(\mu) \int_{\Re(r)=\mathfrak{r}} \abs{G^*((1+2\Delta,\Delta)+r,-\mu)} \abs{dr},
\end{align*}

Since $\mathfrak{u}_i\pm\Re(\mu_j), \mathfrak{r}_i\pm\Re(\mu_j) > -1-\epsilon$, we may apply \lemref{lem:GStarIntegral}, and each of the above products is at most $M_\text{sym}\paren{\kappa,\delta; \mu}$ where now
\begin{align*}
	\delta =& 1-\Delta+\delta_u+\min \Biggl\{\frac{2(1+2\Delta+\mathfrak{r}_1)-\Re(\mu_2)}{2},\frac{1+2\Delta+\mathfrak{r}_1+\Re(\mu_1)}{2}, \\
	& \qquad \frac{2(\Delta+\mathfrak{r}_2)+\Re(\mu_1)}{2}, \frac{\Delta+\mathfrak{r}_2-\Re(\mu_2)}{2}\Biggr\},
\end{align*}
and the $\kappa$ and $\delta_u$ parameters are given in \tableref{tab:uandrBoundParams}.
The bounds may be expressed in this manner because the worst bound happens exactly when the least exponent $\delta$ is on the least difference $\abs{\mu_i-\mu_j}$.
We have sacrificed some efficiency for symmetry in the bounds for the $u$ residues, but it is unlikely that one could exploit what we lost in any case.
This is sufficient to complete the proof of \propref{prop:Bounds}.

\begin{table}
\begin{tabular}{r|cc}
	Form & $\kappa$ & $\delta_u$ \\
	\hline
	$G^* G^*$ & $\mathfrak{u}_1+\mathfrak{u}_2+\mathfrak{r}_1+\mathfrak{r}_2+3$ & $\min \set{\frac{2\mathfrak{u}_1+\Re(\mu_1)}{2},\frac{\mathfrak{u}_1-\Re(\mu_2)}{2},\frac{2\mathfrak{u}_2-\Re(\mu_2)}{2},\frac{\mathfrak{u}_2+\Re(\mu_1)}{2}}$ \\

	$G^*_l G^*$ & $\mu_1+\mathfrak{u}_2+\mathfrak{r}_1+\mathfrak{r}_2+\frac{5}{2}$ & $\Min{\frac{\Re(\mu_1-\mu_2)-1}{2}, \frac{\Re(\mu_1-\mu_3)-1}{2}, \mathfrak{u}_2+\frac{\Re(\mu_2+\mu_3)}{2}}$ \\

	$G^*_r G^*$ & $\mathfrak{u}_1-\mu_2+\mathfrak{r}_1+\mathfrak{r}_2+\frac{5}{2}$ & $\Min{\frac{\Re(\mu_1-\mu_2)-1}{2}, \mathfrak{u}_1-\frac{\Re(\mu_1+\mu_3)}{2}, \frac{\Re(\mu_3-\mu_2)-1}{2}}$ \\

	$G^*_b G^*$ & $\mu_1-\mu_2+\mathfrak{r}_1+\mathfrak{r}_2+2$ & $\Min{\frac{\Re(\mu_1-\mu_2)-1}{2}, \frac{\Re(\mu_1-\mu_3)-1}{2}, \frac{\Re(\mu_3-\mu_2)-1}{2}}$
\end{tabular}
\caption{Parameters for bounding the $u$ and $r$ integrals.}
\label{tab:uandrBoundParams}
\end{table}

\subsection{The decay hypotheses of \thmref{thm:KuznetsovSimplification} and \thmref{thm:PartialInversion}}
We will frequently use the notation \eqref{eq:Mpolydef}, \eqref{eq:Msymdef}, and \eqref{eq:Lambdapolydef} for polynomial bounds in $\mu$ throughout the remainder of the paper.
The content of \thmref{thm:MeanValueEstimates} is essentially that the spectral side converges absolutely when the trivial element term does; in other words, convergence of the integral
\begin{align*}
	\int_{\Re(\mu)=\eta} \abs{\hat{k}(\mu)} \abs{\mu_1-\mu_2} \abs{\mu_1-\mu_3} \abs{\mu_2-\mu_3} \abs{d\mu}
\end{align*}
for $\abs{\eta_i} \le \frac{5}{14}$ gives absolute convergence of the spectral side as well as the trivial term.
By \propref{prop:JwHoloAndBd}, increasing the exponents on the $\abs{\mu_i-\mu_j}$ terms to $M_\text{sym}\paren{\frac{7}{2},\frac{1-4\Delta}{8};\mu}$ gives a sufficient condition for convergence of the intermediate and long element terms in \thmref{thm:KuznetsovSimplification}, hence justifies our use of \eqref{eq:KuzSimplkBd} as the convergence hypothesis on $\hat{k}$ and completes that theorem.

To complete the proof of \thmref{thm:PartialInversion}, we again need to demonstrate that the hypothesis $\hat{f}(q) \ll \abs{q_1 q_2 (q_2-q_1)}^{-4}$ is sufficient for absolute convergence of all of the relevant terms.
For the spectral side and the trivial term, we start with \eqref{eq:KuzSimplkBetterBd} at small $\Delta$ and apply Stirling's formula to $K_{w_l}$ to obtain
\begin{align*}
	& \abs{K_{w_l}(\mu)} \asymp \\
	& M_\text{poly}\paren{\frac{1}{2}+\Re(\mu_1-\mu_2), \frac{1}{2}+\Re(\mu_1-\mu_3), \frac{1}{2}+\Re(\mu_3-\mu_2);\mu},
\end{align*}
for $\Re(\mu)$ constant and away from the lines
\begin{align*}
	& \Re(\mu_1)-\Re(\mu_2),\Re(\mu_1)-\Re(\mu_3),\Re(\mu_3)-\Re(\mu_2) \in \\
	& \set{n, \pm(1+\Delta+n):n\ge 0},
\end{align*}
so we desire convergence of the integral
\begin{align*}
	L_1 :=& \int_{\Re(q) = -\frac{5}{14}-\epsilon} \abs{q_1}^{-4} \abs{q_2}^{-4} \abs{q_2-q_1}^{-4} \\
	& \qquad \abs{q_1+q_2}^{\frac{31}{14}+2\epsilon} \abs{2q_1-q_2}^{\frac{6}{7}+\epsilon} \abs{2q_2-q_1}^{\frac{6}{7}+\epsilon} \\
	& \qquad \int_{\Re(\mu)=\eta} \frac{\abs{\mu_1-\mu_2}^{2+100\epsilon} \abs{\mu_1-\mu_3}^{2+100\epsilon} \abs{\mu_2-\mu_3}^{1+100\epsilon}}{\prod_{j=1}^3 \abs{q_1-\mu_j}\abs{q_2+\mu_j}} \\
	& \qquad \abs{k_\text{conv}(\mu,q)} \, \abs{d\mu} \, \abs{dq},
\end{align*}
by symmetry in $\mu$.
From \propref{prop:Bounds}, we note that this also implies absolute convergence of the $F_1$ error term.
To simplify the above integral, we choose
\begin{align}
\label{eq:kconv}
	k_\text{conv}(\mu,q) &= \paren{\frac{\paren{5+q_1}\paren{5-q_2}\paren{5-q_1+q_2}}{\paren{5+\mu_1}\paren{5+\mu_2}\paren{5+\mu_3}}}^{\frac{145}{72}+100\epsilon}.
\end{align}
We call this function $k_\text{conv}$ because without it the above integral would not converge -- the total power on $\mu_1$ and $\mu_2$ is not below $-1$.

Applying the same logic to the remaining error terms, it is sufficient to consider convergence of the integrals
\begin{align*}
	L_2 :=& \int_{\Re(\mu)=\paren{-\frac{1}{2}-4\epsilon,\frac{1}{2}+\epsilon}} \int_{\Re(q_2) = -6\epsilon} \abs{\mu_1}^{-4} \abs{q_2}^{-4} \abs{q_2-\mu_1}^{-4} \\
	& \qquad M_\text{sym}\paren{\frac{5}{2},0;\mu} \frac{\abs{q_2+\mu_1}^{1+4\epsilon} \abs{2\mu_1-q_2}^{\frac{3}{2}+5\epsilon} \abs{2q_2-\mu_1}^{-\epsilon}}{\abs{q_2+\mu_2}\abs{q_2+\mu_3}} \\
	& \qquad \abs{k_\text{conv}(\mu,q)} \, \abs{dq_2} \, \abs{d\mu},
\end{align*}
\begin{align*}
	L_3 :=& \int_{\Re(\mu) = (-\epsilon, \epsilon)} \abs{\mu_1}^{-4} \abs{\mu_2}^{-4} \abs{\mu_3}^{-4} \\
	& \qquad \abs{\mu_1-\mu_2}^{\frac{3}{2}+100\epsilon} \abs{\mu_1-\mu_3}^{\frac{5}{4}+100\epsilon} \abs{\mu_3-\mu_2}^{\frac{5}{4}+100\epsilon} \abs{d\mu}.
\end{align*}
Here convergence of the $L_2$ integral is sufficient to show absolute convergence of $F_2$ and $F_3$ (by symmetry), and $L_3$ gives the absolute convergence of the remaining error terms.
The term $M_\text{sym}\paren{\frac{5}{2},0;\mu}$ comes from \eqref{eq:BetterJwlBound} which is proved by the above reasoning for the $E_{w_l,j}$ functions with the contours $\Re(u) = -\frac{1}{2}-4\epsilon$, $\Re(r) = \paren{\frac{1}{2}-4\epsilon,100\epsilon}$, $\mathfrak{s} = -\frac{1}{2}-\epsilon$, $\mathfrak{t} = \epsilon$, and taking the worst bound from each of the $G^* G^*$, $G^*_l G^*$, $G^*_r G^*$, and $G^*_b G^*$ configurations.
The exponents in $L_3$ are derived from
\[ \abs{K_{w_l}(\mu)}^{-1} \ll \abs{\mu_1-\mu_2}^{\frac{1}{2}} \abs{\mu_1-\mu_3}^{\frac{1}{4}} \abs{\mu_3-\mu_2}^{\frac{1}{4}}, \]
\[ \abs{E_{w_l,j}} \ll \abs{y_1 y_2}^{\frac{1}{2}+\epsilon} \abs{\mu_1-\mu_2}^{1+100\epsilon} \abs{\mu_1-\mu_3}^{1+100\epsilon} \abs{\mu_3-\mu_2}^{1+100\epsilon}. \]

The convergence of the integrals $L_1$, $L_2$ and $L_3$ completes \thmref{thm:PartialInversion}.
Said convergence is easiest to obtain by noting that
\[ \abs{\mu_1-\mu_2}^a \abs{\mu_1-\mu_3}^b \abs{\mu_2-\mu_3}^c \ll \abs{\mu_1}^d \abs{\mu_2}^e \abs{\mu_3}^f, \]
for any $d+e+f=2(a+b+c)$, $a,b,c,d,e,f\ge0$, $d,e,f \le a+b+c$, which follows from $\abs{\mu_1-\mu_2} \ll \abs{\mu_1} \abs{\mu_2}$, etc. and $\mu_1-\mu_2 = 2\mu_1+\mu_3=-\mu_3-2\mu_2$, etc.
A similar bound applies in the $q$ variables.
Using such a bound is clearly wasteful, but this simplifies the proof immensely.

For example, in the $L_1$ integral, after expanding $k_\text{conv}$, we remove the positive-exponent $q$ terms and part of $M_\text{sym}\paren{5,1}$ with the respective bounds
\[ \paren{\prod_{j=1}^3 \abs{q_1-\mu_j} \abs{\mu_j-q_2}}^{\frac{1}{3}-90\epsilon} \abs{q_1}^{\frac{41}{21}+200\epsilon} \abs{q_2}^{\frac{41}{21}+200\epsilon} \abs{q_2-q_1}^{\frac{41}{21}+200\epsilon} \]
and
\[ \abs{\mu_1-\mu_2}^{\frac{5}{6}+10\epsilon} \abs{\mu_1-\mu_3}^{\frac{5}{6}+10\epsilon}\abs{\mu_1-\mu_2}^{\frac{1}{3}+10\epsilon} \ll \prod_j \abs{\mu_j}^{\frac{4}{3}+20\epsilon} \]
giving
\begin{align*}
	L_1 \ll& \int_{\Re(\mu)=\eta} \abs{\mu_1}^{-\frac{49}{72}-80\epsilon} \abs{\mu_2}^{-\frac{49}{72}-80\epsilon} \abs{\mu_3}^{-\frac{49}{72}-80\epsilon} \\
	& \qquad \int_{\Re(q) = -\frac{5}{14}-\epsilon} \frac{\abs{\mu_1-\mu_2}^{\frac{2}{3}+90\epsilon} \abs{\mu_1-\mu_3}^{\frac{2}{3}+90\epsilon} \abs{\mu_2-\mu_3}^{\frac{2}{3}+90\epsilon}}{\paren{\prod_{j=1}^3 \abs{q_1-\mu_j}\abs{q_2+\mu_j}}^{\frac{2}{3}+90\epsilon}} \, \abs{d\mu} \, \abs{dq},
\end{align*}
Applying H\"older to the $q$ integrals gives 3 integrals of the form
\begin{align*}
	& \paren{\int_{\Re(q) = \mathfrak{q}} \paren{\abs{q_1-\mu_j}\abs{q_1-\mu_k}}^{-1-135\epsilon} \paren{\abs{q_2-\mu_j}\abs{q_2-\mu_k}}^{-1-135\epsilon} \, \abs{dq}}^{\frac{1}{3}} \\
	& \ll \abs{\mu_k-\mu_j}^{-\frac{2}{3}-90\epsilon},
\end{align*}
so
\begin{align*}
	L_1 \ll& \int_{\Re(\mu)=\eta} \abs{\mu_1}^{-\frac{49}{72}-80\epsilon} \abs{\mu_2}^{-\frac{49}{72}-80\epsilon} \abs{\mu_3}^{-\frac{49}{72}-80\epsilon} \abs{d\mu} \ll 1.
\end{align*}

\subsection{\texorpdfstring{The $G$ Function}{The G Function}}
\label{sec:GFuncBd}
The bound for the residue $G^*_b(\mu)$ in \lemref{lem:GStarIntegral} part \textit{a} is simply from applying Stirling's formula.
The residue $G^*_l(u_2,\mu)$ has exponential decay in $\Im(u_2)$ unless $\Im(\mu_2) > \Im(-u_2) > \Im(\mu_3)$ (up to permutation of $(\mu_2, \mu_3)$, for fixed $\Re(u_2)$), thus it integrates much like a beta function, and the bound in \lemref{lem:GStarIntegral} part \textit{b} follows by the same logic, similarly for part \textit{c}.

For $\Re(u)$, $\Re(\mu)$ fixed, applying Stirling's formula gives exponential decay in $u_1$ or $u_2$ for $G^*$ unless, up to permutation of $\mu$ or $(u_1,-u_2)$,
\[ \Im(\mu_1)>\Im(u_1)>\Im(\mu_2)>\Im(-u_2)>\Im(\mu_3), \]
in which case, we have
\begin{align*}
	G^*(u,\mu) &\ll \frac{\abs{u_1+u_2}^{\frac{1-\Re(u_1+u_2)}{2}} \prod_{i=1}^3 \abs{u_1-\mu_i}^{\frac{\Re(u_1-\mu_i)-1}{2}} \abs{u_2+\mu_i}^{\frac{\Re(u_2+\mu_i)-1}{2}}}{\Lambda_\text{poly}(\mu)}.
\end{align*}

For integrals of the $G^*$ function, it is sufficient to prove the following lemma:
\begin{lem}
	Suppose $a_i, b_i > -1-\epsilon$, $a_2+b_2+c > -2-\epsilon$, and $v_1<v_2<v_3$ with $v_3-v_2>v_2-v_1$, then
	\begin{align*}
		\int_{v_1}^{v_2} \int_{v_2}^{v_3} \abs{1+i(u_1-u_2)}^c \prod_{i=1}^3 \abs{1+i(u_1-v_i)}^{a_i} \abs{1+i(u_2-v_i)}^{b_i} du_2 \, du_1 \\
		\ll \abs{1+i(v_3-v_1)}^{\frac{\kappa-\delta}{2}+\epsilon} \abs{1+i(v_3-v_2)}^{\frac{\kappa-\delta}{2}+\epsilon} \abs{1+i(v_2-v_1)}^{\delta+\epsilon},
	\end{align*}
	where
	\begin{align*}
		\kappa =& a_1+a_2+a_3+b_1+b_2+b_3+c+2, \\
		\delta =& \Min{a_1+a_2+1, a_1+a_2+b_1+c+1}.
	\end{align*}
\label{lem:ElementaryGIntegral}
\end{lem}

Applying this to the $G^*$ function, we have
\begin{align*}
	\kappa =& \mathfrak{u}_1+\mathfrak{u}_2 - \frac{1}{2}, \\
	\delta =& \Min{\frac{\kappa}{3}, \mathfrak{u}_1+\frac{\Re(\mu_3)}{2}, \frac{\mathfrak{u}_1-\Re(\mu_2)}{2}},
\end{align*}
and the conditions become
\begin{align*}
	& a_i, b_i > -1-\epsilon \Leftrightarrow \mathfrak{u}_1-\Re(\mu_i), \mathfrak{u}_2+\Re(\mu_i) > -1-2\epsilon, \\
	& a_2+b_2+c = -\frac{1}{2} > -2-\epsilon.
\end{align*}
Permuting $\mu$ and interchanging $(\mathfrak{u}_1, \mu) \leftrightarrow (\mathfrak{u}_2,-\mu)$ as necessary, we obtain \lemref{lem:GStarIntegral} part \textit{d}.
We do not give the proof that the integral over the region of exponential decay satisfies the same bound, but if one conditions on $v_2-v_1$ smaller or larger than
\[ \delta = 10\paren{\abs{a_1}+\abs{a_2}+\abs{a_3}+\abs{b_1}+\abs{b_2}+\abs{b_3}+\abs{c}+3}\log \abs{1+i(v_3-v_1)}, \]
the proof in the first case is essentially identical to the case $v_2-v_1 < 1 \le v_3-v_2$ below, and in the second case we need only extend the region of integration by $\delta \ll \abs{1+i(v_3-v_1)}^\epsilon$.

\begin{proof}[Proof of \lemref{lem:ElementaryGIntegral}]
We will repeatedly use the fact $v_3-v_1 =(v_3-v_2)+(v_2-v_1) \asymp v_3-v_2$.

If $v_3-v_2 < 1$, the result is obvious; if $v_2-v_1 < 1 \le v_3-v_2$, the integral reduces to
\begin{align*}
	&\abs{1+i(v_3-v_2)}^{a_3} \int_{v_2}^{v_3} \abs{1+i(u_2-v_2)}^{c+b_1+b_2} \abs{1+i(u_2-v_3)}^{b_3} du_2 \\
	&\ll \abs{1+i(v_3-v_2)}^{a_3 + \Max{c+b_1+b_2,b_3,c+b_1+b_2+b_3+1}} \\
	&= \abs{1+i(v_3-v_2)}^{a_3+\Max{b_3,c+b_1+b_2+b_3+1}+\epsilon},
\end{align*}
since $\delta \le \Min{a_1+a_2+1,a_1+a_2+b_1+b_2+c+2}$, we have
\[ \kappa-\delta \ge \Max{a_3+b_1+b_2+b_3+c+1,a_3+b_3}, \]
and the result follows.

Now assume $v_2-v_1 \ge 1$.
We split the integral into three parts: The first, call it $I_1$, with $u_2 > \frac{v_2+v_3}{2}$, the second with $u_1 < \frac{v_1+v_2}{2}$, and the third over the remaining region.
For the first integral, we send $u_2 \mapsto u_2 + v_3$, then
\begin{align*}
	& \int_{-\frac{v_3-v_2}{2}}^0 \abs{1+i(u_1-v_3-u_2)}^c \abs{1+i(u_2+v_3-v_1)}^{b_1} \\
	& \qquad \abs{1+i(u_2+v_3-v_2)}^{b_2} \abs{1+i u_2}^{b_3} du_2 \\
	& \asymp \abs{1+i(u_1-v_3)}^c \abs{1+i(v_3-v_1)}^{b_1} \abs{1+i(v_3-v_2)}^{b_2} \int_{-\frac{v_3-v_2}{2}}^0 \abs{1+i u_2}^{b_3} du_2 \\
	& \ll \abs{1+i(u_1-v_3)}^c \abs{1+i(v_3-v_1)}^{b_1} \abs{1+i(v_3-v_2)}^{b_2+\Max{0,b_3+1}},
\end{align*}
since $v_3-u_1 \ge v_3-v_2 > v_2-v_1$.
Then
\begin{align*}
	I_1 &\ll \abs{1+i(v_3-v_1)}^{b_1} \abs{1+i(v_3-v_2)}^{b_2+\Max{0,b_3+1}} \\
	& \qquad \int_{v_1}^{v_2} \abs{1+i(u_1-v_1)}^{a_1} \abs{1+i(u_1-v_2)}^{a_2} \abs{1+i(u_1-v_3)}^{a_3+c} du_1,
\end{align*}
and splitting this integral at $\frac{v_1+v_2}{2}$, we obtain
\begin{align*}
	& \int_0^{\frac{v_2-v_1}{2}} \abs{1+i u_1}^{a_1} \abs{1+i(u_1+v_1-v_2)}^{a_2} \abs{1+i(u_1+v_1-v_3)}^{a_3+c} du_1 \\
	& \ll \abs{1+i(v_3-v_1)}^{a_3+c} \abs{1+i(v_1-v_2)}^{a_2+\Max{0,a_1+1}},
\end{align*}
and similarly for
\begin{align*}
	& \int_{-\frac{v_2-v_1}{2}}^0 \abs{1+i(u_1+v_2-v_1)}^{a_1} \abs{1+i u_1}^{a_2} \abs{1+i(u_1+v_2-v_3)}^{a_3+c} du_1 \\
	& \ll \abs{1+i(v_2-v_1)}^{a_1+\Max{0,a_2+1}} \abs{1+i(v_2-v_3)}^{a_3+c}
\end{align*}
giving
\begin{align*}
	I_1 &\ll \abs{1+i(v_3-v_2)}^{a_3+b_1+b_2+b_3+c+1+\epsilon} \abs{1+i(v_1-v_2)}^{a_1+a_2+1+\epsilon},
\end{align*}
using $v_3-v_1 \asymp v_3-v_2$ and the hypotheses on $a$ and $b$.
The remaining terms are similar.

Then $\delta$ is the minimum exponent of $\abs{1+i(v_2-v_1)}$ and $\kappa$ is the sum of the two exponents, which is the same in every case; we split the remaining $\kappa-\delta$ evenly between $\abs{v_3-v_2} \asymp \abs{v_3-v_1}$.
\end{proof}

\section{Acknowledgements}
The author would like to thank Prof. W. Duke for his guidance along the way, Prof. X. Li for a copy of her ``Kloostermania'' notes and accompanying advice in addition to some timely warnings, Prof. V. Blomer for the Kim-Sarnak result and other helpful comments, and Prof. P. Sarnak for bringing the Goldfeld-Sarnak result to his attention and some history on the Kloostermania.

\bibliography {bibliography} {}
\bibliographystyle{plain}

\end{document}